\documentclass[a4paper, 11pt]{amsart}
\usepackage[T1]{fontenc}
\usepackage[utf8]{inputenc}
\usepackage{amsmath}
\usepackage{amsfonts}
\usepackage{amssymb}
\usepackage{amsthm}
\usepackage{graphicx}
\usepackage[english]{babel}
\usepackage{version} 
\usepackage{nicefrac}
\usepackage{tipa} 
\usepackage{hyperref}

\usepackage{xcolor}
\usepackage{soul}
\theoremstyle{definition}
\newtheorem{defin}{Definition}[section]
\newtheorem{ex}[defin]{Example}
\theoremstyle{plain}
\newtheorem{theo}[defin]{Theorem}
\newtheorem{lemma}[defin]{Lemma}
\newtheorem{obs}[defin]{Remark}
\newtheorem{prop}[defin]{Proposition}
\newtheorem{cor}[defin]{Corollary}
\newtheorem*{conj}{Conjecture}

\newtheorem*{theo-intro}{Theorem}
\newtheorem{theorem}{Theorem}
\newtheorem{corollary}[theorem]{Corollary}

\renewenvironment{abstract}
{\par\noindent\textbf{\abstractname.}\ \ignorespaces}
{\par\medskip}

\title{Thin actions on CAT(0)-spaces}
\author{Nicola Cavallucci}
\thanks{N.Cavallucci is partially supported by the SFB/TRR 191, funded by the DFG.}
\author{Andrea Sambusetti}
\thanks{A. Sambusetti is member of the Differential Geometry section of the GNSAGA-INdAM.}
\address{Andrea Sambusetti, Dipartimento di Matematica “Guido
	Castelnuovo”, SAPIENZA Universita di Roma, Piazzale Aldo Moro 5, I-00185 `
	Roma.}
\address{Nicola Cavallucci, Karlsruhe Institute of Technology, Engelstrasse 2, D-76128 Karlsruhe}
\email{n.cavallucci23@gmail.com, sambuset@mat.uniroma1.it}
\date{}

\begin{document}
\maketitle

\footnotesize
\begin{abstract}
	We study groups of isometries of packed, geodesically complete, CAT$(0)$-spaces for which the systole at every point is smaller than a universal constant depending only on the packing, deducing strong rigidity results. We show that if a space as above has some negative curvature behaviour then it cannot support a thin action: this generalizes the classical Margulis Lemma to a broader class of spaces.
\end{abstract}
\normalsize

%\tableofcontents

\section{Introduction}

The classical Margulis-Heintze Lemma is a cornerstone in the study of negatively curved manifolds: it  asserts that there exists some constant $\varepsilon(n,\kappa)>0$ such that for any torsionless lattice $\Gamma$ of a complete, Riemannian $n$-dimensional manifold $X$ with sectional curvature $-\kappa^2\leq K(X)<0$ there exists a point $x$ such that   $$\text{sys}(\Gamma,x):=\inf_{g \in \Gamma^\ast}d(x,gx) > \varepsilon(n,\kappa),$$ 
where $\Gamma^\ast$ is the subset of nontrivial elements of $\Gamma$ (see  \cite{Hei},  \cite{BGS13}, \cite{BZ}). Accordingly, for any discrete group $\Gamma$ acting by isometries on a metric space $X$ we will call   {\em diastole} of the action (or of the quotient $M=\Gamma \backslash X$), the value 
 $$  \text{dias}(\Gamma, X) = \sup_{x \in X} \inf_{g \in \Gamma^\ast} d(x,gx).$$  
We say that the action of $\Gamma$ is \emph{$\varepsilon$-thin} if $\text{dias}(\Gamma,X)<\varepsilon$ (since the subset of points $x$ where $\text{sys}(\Gamma,x)<\varepsilon$ is classically known as the $\varepsilon$-thin subset of X).\\ \noindent Buyalo first, in dimension smaller than $5$  \cite{Buy1}-\cite{Buy2}, and Cao-Cheeger-Rong  later, in any dimension \cite{CCR01},  studied the possibility of extending the Margulis-Heintze Lemma to {\em uniform lattices of non-positively curved} manifolds. 
They proved the following alternative: either the diastole is bounded below by a universal positive constant $\eta(n,\kappa) > 0$, or the action of $\Gamma$ is $\eta$-thin and the manifold $\Gamma\backslash X$ admits  a so-called {\em abelian local splitting structure}. This is, roughly speaking, a decomposition of $X$ into a union of minimal sets of hyperbolic isometries with the additional property that if two minimal sets intersect then the corresponding isometries commute.
%
%
% into a locally finite family of closed convex spaces Min$({\cal C}_i)$ which are the minimum sets of families ${\cal C}_i \subset \Gamma$ of hyperbolic isometries, which split  as Riemannian  products  Min$({\cal C}_i) \cong D_i \times {\mathbb R}^{d_i}$ (with $d_i \geq1$),  where ${\cal C}_i$ acts trivially  on the factor $D_i$ and   cocompactly by  translations on $ {\mathbb R}^{d_i}$, and with  the additional  property that if Min$({\cal C}_i) \cap \textup{Min}({\cal C}_j) \neq \emptyset$ then the elements of ${\cal C}_i$ and ${\cal C}_j$ commute.  
%When this last condition is replaced by the condition that the subgroup  of $\Gamma$ generated by ${\cal C}_i$ and ${\cal C}_j$  is virtually abelian, we call this   a   {\em virtually abelian local splitting structure} for $X$;  the definition can be rephrased in terms of the quotient space $M$, see subsection \S\ref{v.ab.localsplittingstructures} for   precise definitions.   

\begin{ex}[\cite{Gro78}, Section 5 and \cite{Buy81}, Section 4]
	\label{example-Gromov-Buyalo} 
	Consider two copies $\Sigma_1, \Sigma_2$ of a hyperbolic surface with connected, geodesic boundary of length $\ell$, then take the products $M_i= \Sigma_i \times S^1$ with a circle of length $\ell$, and glue $M_1$ to $M_2$ by means of an isometry  $\varphi$ of the boundary tori $\partial M_i= \partial \Sigma_i \times S^1$ which interchanges the circles $\partial \Sigma_i$ with  $S^1$.   This yields a nonpositively curved $3$-manifold $M$ (a {\em graph manifold}) with sectional curvature $-1 \leq k(M) \leq 0$ whose  diastole can be chosen arbitrarily small, by taking $\ell \rightarrow 0$. This manifold is the prototypical example of a nonpositively curved manifold with an abelian local splitting structure: calling $g_1, g_2$ the hyperbolic isometries of $X=\widetilde M$ corresponding to the boundary loops $\partial \Sigma_i$, the universal cover $\widetilde X$ is the union of the two closed, minimum sets Min$(g_i) \cong \widetilde  \Sigma_i \times {\mathbb R}$ (where $g_i$ acts as $id$ times a translation) and Min$(g_1) \cap $Min$(g_2)$ is the gluing torus, on which $\langle g_1,g_2 \rangle$ acts on as an abelian group.
\end{ex}

\vspace{1mm}
In this paper, we will generalize this result to groups acting (not necessarily cocompactly) on CAT(0)-spaces and we draw from it a 
%some rigidity consequences  about 
%a structural 
description of spaces with thin actions. In the framework of CAT(0)-spaces,  the lower sectional curvature bound appearing in the assumptions of the classical Margulis-Heintze Lemma can  be naturally replaced by a uniform, local packing condition: we say that a metric space $(X,d)$ satisfies the {\em $P_0$-packing condition at scale $r_0 > 0$}
% (or that $X$ is {\em $(P_0,r_0)$-packed}, for short)
 if all balls of radius $3r_0$ in $X$ contain at most $P_0$ points that are $2r_0$-separated from each other.  
In \cite{CavS20} is proved that, for complete and geodesically complete CAT(0)-spaces, a packing condition at some scale $r_0$ yields an explicit, uniform control of the packing function at any other  scale $r$: therefore, for these spaces, this condition is  equivalent to similar conditions which have been considered by other authors with different names (``uniform compactness of the family of $r$-balls'' in \cite{Gr81}; ``geometrical boundedness'' in \cite{DY}, etc.). 
This packing condition should be thought as a macroscopic  replacement of a lower bound on the curvature:  for Riemannian manifolds it is   strictly weaker than a lower bound of the Ricci curvature, and for general metric spaces it is   weaker than asking that $X$ is  a space with curvature bounded below in the sense of Alexandrov (see [BCGS17], Sec.3.3, for different examples and comparison between packing and curvature conditions). 
Although much weaker than a curvature bound, the  packing condition implies, by   Breuillard-Green-Tao's work, an analogue  of the celebrated Margulis' Lemma for a general metric space  $X$: {\em there exists a constant $\varepsilon_0>0$, only depending on the packing constants $P_0,r_0$,
%(only depending on the $n$-dimensional macroscopic scalr curvature bound) 
such that for every discrete group of isometries $\Gamma$ of   $X$ 
%which is  $P_0$-packed at scale $r_0$, 
the 	$\varepsilon_0$-almost stabilizer $\Gamma_{\varepsilon_0} (x)$ of any point $x$ is virtually nilpotent}  (cp. \cite{BGT11} and Section \S\ref{subsec-packed,margulis} below for details).
% that is, the elements of $\Gamma$ which displace $x$ less than $\varepsilon_0$ generate a virtually nilpotent group.
 We call this $\varepsilon_0=\varepsilon_0 (P_0, r_0)$ the {\em Margulis constant}, since it plays the role of the classical Margulis constant in this metric setting.

\noindent There are a lot of non-manifolds examples in the class of packed, CAT(0)-spaces: limits of Hadamard manifolds with curvature bounded below;  simplicial complexes with locally constant curvature and bounded geometry (also called $M_\kappa$-complexes, cp. [BH13] and \cite{CavS20}), Euclidean and hyperbolic buildings with bounded geometry in particular; nonpositively curved cones over compact CAT(1)-spaces,  etc.
\vspace{2mm}

\noindent  The first theorem of this paper generalizes Cao-Cheeger-Rong's result to this  broader setting, taking into account also non-compact actions of groups, possibly  with  torsion:
 
\begin{theorem}[Extract from Theorem \ref{theo-alternative}] ${}$
\label{maintheorem}

\noindent There exist two functions  $\lambda_0= \lambda_0 (P_0, r_0)$, $\eta_0=(P_0, r_0)>0$  with the following property:
for any  complete, geodesically complete, \textup{CAT}$(0)$-space $X$  which is $(P_0,r_0)$-packed,  and any   almost-cocompact discrete group of isometries $\Gamma$ of $X$, if $\Gamma$ is $\eta_0$-thin then $X$ can be decomposed as
	
	 \begin{equation}\label{UMin}
	 X=\bigcup_{g\in \Sigma_ {\lambda_0}} \textup{Min}(g)
	 \end{equation}
	 
\noindent  where  $\Sigma_{\lambda_0}$ is the subset of   $g \in \Gamma^*$   with minimal displacement $\ell(g)\leq \lambda_0$.
 
\end{theorem}

\noindent By {\em almost-cocompact} group, we mean  a group which acts cocompactly on the $\delta$-thick subsets of $X$. The same notion was introduce in \cite{BGS13}, §8.4 with the notation InjRad$(\Gamma\backslash X) \to 0$. In particular, cocompact groups or groups acting with quotient of finite measure are almost cocompact (see Section \S2 for the definition and the natural measure of CAT(0)-spaces under consideration).
\vspace{2mm}
 
\noindent The proof of Theorem \ref{maintheorem} is  different from that for manifolds given in \cite{CCR01}: actually,  the proof of Cao-Cheeger-Rong is based on the fact that hyperbolic elements {\em stabilize} (i.e., for any $s > 0$ and for any such $g$ there exists $n$ such that Min$(g^n)= $Min$(g^{nk})$ for all $0<k\leq s$), a property which does not hold in general for CAT(0)-spaces, as the following easy example shows.
\begin{ex}
	\label{ex-stability-isometries}
	Let $X$ be the product of the regular tree $T$ of valency $3$ and $\mathbb{R}$. Let $g_T$ be an elliptic isometry of $T$ with infinite order and such that $\text{Fix}(g_T^n) \neq \text{Fix}(g_T^k)$ for all $n,k\in \mathbb{N}$, $n\neq k$. Let $\tau$ be a translation of $\mathbb{R}$. The isometry $g=(g_T,\tau)$ of $X$ is hyperbolic but no power of $g$ stabilizes: indeed for every $n,k\in \mathbb{N}$, $n\neq k$ it holds $\text{Min}(g^{n}) \neq \text{Min}(g^k)$.
\end{ex}

A major difficulty for this is that, in this extended  metric setting,  there exist  different isometries which coincide on open sets. This problem is related to the existence of non-trivial isometry whose set of fixed points has non-empty interior. For many applications of Theorem \ref{maintheorem} we will exclude this case, restricting the attention to what we called \emph{slim groups} (we refer to Section \ref{subsec-isometries} for more details). By definition two isometries of a slim group coinciding on an open set must coincide globally, this gives a better control on the geometry of the group action. Every torsion-free group is trivially slim, as well as any discrete group acting on a CAT$(0)$-homology manifold (cp. Lemma \ref{lemma-class-manifolds}). Moreover, as opposite to \cite{CCR01}, we also have to deal with parabolic and elliptic isometries, as $\Gamma$ is allowed to have torsion.\\
%; therefore,  Theorem \ref{maintheorem} generalizes Cao-Cheeger-Rong to a much broader setting.
 On the other hand,  our proof has  some similarities with a minimality argument used by Cao-Cheeger-Rong  (see their proof of Theorem 0.3 in   \cite{CCR01}), and  it is inspired by statements of Ballmann-Gromov-Schroeder (Lemma 7.11 and Section 13 of  \cite{BGS13}) holding for {\em analytic} manifolds (we stress that analiticity is crucial in \cite{BGS13} for their argument to work); similar arguments, which  use   analiticity or strictly negative curvature, can be found in  \cite{Gel11}, Section 2 and \cite{BGLS10}, Section 2. The main innovation in our proof relies in the choice of the good minimal point for the function $\Psi$ (cp. Step 3   in the proof of Theorem \ref{theo-alternative}), which allows us to overcome any analyticity or negative curvature assumption. 
\vspace{2mm}

\noindent In Theorem \ref{theo-alternative} we will see a stronger statement than Theorem \ref{maintheorem}: the same alternative holds for any $\lambda$ smaller than the  Margulis' constant $\varepsilon_0$ and an explicit $\eta=\eta(P_0,r_0, \lambda)$.
Here we want just to stress   that the equality (\ref{UMin}), in particular the fact the the union runs over minimum sets of isometries with minimal displacement less than the Margulis constant,  imposes a very strong condition on the topology of $X$. For instance, if   $X$ is a non-elementary Gromov hyperbolic space and $\Gamma$ is torsionless, then $X$ cannot satisfy \eqref{UMin}, and  the diastole is always universally bounded away from zero. This was already proved  in  \cite{BCGS}  in the cocompact case (see proof of Proposition 5.24), and in \cite{CavS20bis} (Corollary 1.4)  for non-cocompact actions. The following corollary yields a simple algebraic obstruction.
  \begin{corollary}
	\label{cor-ZxZ}
	Let $X$ be a complete, geodesically complete, \textup{CAT}$(0)$-space which is $P_0$-packed at scale $r_0$. Assume that $\Gamma$  is a discrete, slim, almost-cocompact group of isometries of $X$. If  $\textup{dias}(\Gamma , X) <\eta_0$, then 
		\begin{itemize}
			\item[(i)] either $\Gamma$ contains a subgroup isomorphic to $\mathbb{Z} \times \mathbb{Z}$
			\item[(ii)] or $X$ splits a non-trivial Euclidean factor.
		\end{itemize}
	If moreover $\Gamma$ is cocompact then (ii) can be replaced by 
	\begin{itemize}
		\item[(ii')] $X=\mathbb{R}$.	
	\end{itemize} 
	 \noindent (The constant $\eta_0$ here and in the following  is the same as  in Theorem  \ref{maintheorem}.)
\end{corollary}

 The following theorem generalizes the lower bounds  for dias$(\Gamma, X)$ of  \cite{BCGS} and \cite{CavS20bis}, and completely determines the structure of packed (complete, geodesically complete) CAT(0)-spaces with diastole small enough in some remarkable cases: 
  %excludes the existence of virtually abelian splitting structures i 

  \begin{theorem}[Extract from Proposition  \ref{prop-dimension-1}  and  Corollaries \ref{corollary-CAT-epsilon} \& \ref{cor-visibility}] ${}$
\label{teor-dias}
\noindent  Let $X$ be a complete, geodesically complete, \textup{CAT}$(0)$-space which is $P_0$-packed at scale $r_0$.  Assume moreover that one of the following  holds: 	
\begin{itemize}
		\item[(i)]   $X$ has a point of dimension $1$,
		  \item[(ii)] or $X$ has a point with an open neighbourhood  which is \textup{CAT}$(-\varepsilon)$ for some $\varepsilon > 0$,
		    \item[(iii)] or $X$ is a visibility space.	  
	\end{itemize}
If $X$ has a discrete, slim,  almost-cocompact group of isometries   $\Gamma$ such that
 $\textup{dias}(\Gamma , X) <\eta_0$, then $X$ is isometric to $ \mathbb{R}$.
\end{theorem}

\vspace{-2mm}
\noindent (See Section \ref{subsec-cat(0)} for the notion of {\em points of dimension $k$}  and generalites on the dimension for CAT(0)-spaces).

\vspace{2mm}
The following is another  structural consequence of a small diastole for CAT(0)-spaces in low dimension.

\begin{theorem}
 	\label{teor-dimension-2}
	Let $X$ be a complete, geodesically complete  \textup{CAT}$(0)$-space  of dimension at most $2$, which is $P_0$-packed at scale $r_0$, and let  $\Gamma$ be a discrete, slim, almost-cocompact group of isometries of $X$. If    $\textup{dias}(\Gamma , X) < \eta_0$  then   $X$ is isometric to $T \times \mathbb{R}$, where $T$ is a tree (possibily a point). If moreover $\Gamma$ is cocompact then it virtually splits as $\Gamma_T \times \mathbb{Z}$, where $\Gamma_T$ (resp. $\mathbb{Z}$) acts discretely and cocompactly on $T$ (resp. $\mathbb{R}$).
\end{theorem}

\noindent This splitting result is false in higher dimensions as showed by Example \ref{example-Gromov-Buyalo}.

\begin{obs}{\em
	This last theorem enlightens two interesting phenomena.  \linebreak
	In dimension $\leq 2$, if the diastole is smaller than the universal constant $\eta_0$ then:
	\begin{itemize}
		\item[(a)] the space is {\em pure dimensional}, that is the dimension of every point is the same;
		\item[(b)] the space is {\em collapsible}, in the following sense: one can define new \textup{CAT}$(0)$ metrics $d_n$ on $X$ in such a way that $(X,d_n)$ is still uniformily packed at scale $r_0$ and  $\Gamma$ still acts by isometries with respect to $d_n$, but   the diastole of $\Gamma$ with respect to the metric $d_n$ is smaller than $\frac{1}{n}$. \\
		Therefore,  if the dias$(\Gamma, X)$ is sufficiently small,  then it can be taken arbitrarily small by changing the metric of $X$ remaining in the class of uniformly packed CAT$(0)$ metrics.
			\end{itemize}

\noindent 	The result of (a) is not true in higher dimensions: it is enough to consider the product $X=Y\times \mathbb{R}$, where $Y$ is any non-pure dimensional, complete, geodesically complete, packed \textup{CAT}$(0)$-space  admitting a cocompact group of isometries $\Gamma_Y$, and  taking  $\Gamma=\Gamma_Y \times \langle \tau_\varepsilon \rangle$,
%acting on $Y\times \mathbb{R}$ being $\Gamma \times \langle \tau_\varepsilon \rangle$, 
where $\tau_\varepsilon$ is a translation by  $\varepsilon$.\\
	On the other hand, we have no clue if (b) is true in higher dimensions.
}
\end{obs}

We remark that Corollary \ref{cor-ZxZ}, Theorem \ref{teor-dias} and Theorem \ref{teor-dimension-2} are false if the group is not slim, as the following example due to A.Lytchak shows.
\begin{ex}
	\label{example-Lytchak}
	Let $X$ be the Cayley graph of the free group on two generators $F(a,b)$. Let $e$ be the identity point on this graph and $\tilde{a}, \tilde{b}$ the axes of the isometries $a$ and $b$. Let $\sigma_a,\sigma_b$ be the reflections with respect to $\tilde{a}$ and $\tilde{b}$. Finally let $\Gamma$ be the group of isometries of $X$ generated by $\lbrace a,b,\sigma_a,\sigma_b\rbrace$. It is clear that every point of $X$ is stabilized by some element of $\Gamma$, so dias$(\Gamma,X) = 0$. The only non-trivial fact is to show that $\Gamma$ is discrete, which is equivalent to say that the stabilizer of $e$ is finite. Let $\mathbb{S}^1 \ast \mathbb{S}^1$ be the quotient of $X$ with respect to the action of $F(a,b)$. Every element of $\Gamma$ defines an isometry of $\mathbb{S}^1 \ast \mathbb{S}^1$ fixing the gluing point. By the standard cover theory each isometry of $\mathbb{S}^1 \ast \mathbb{S}^1$ defines a unique isometry $\tilde{f}$ of $X$ such that $\tilde{f}(e) = e$. So there exists an injective map from Stab$_\Gamma(e)$ to Isom$(\mathbb{S}^1\ast \mathbb{S}^1)$. The latter group is finite, so also Stab$_\Gamma(e)$ is finite.
\end{ex}
In this case the decomposition given by Theorem \ref{maintheorem} contains only elliptic elements and this is why the slim assumption plays a role. The next example shows instead that the decomposition given by Theorem \ref{maintheorem} is much stronger than a decomposition as union of minimal sets of \emph{some} isometries.

\begin{ex}
	Let $Y$ be a flat torus with a loop of length $\ell$ glued at some point. Let $X$ be its universal cover, which is a geodesically complete, CAT$(0)$ space which is $P_0$-packed at scale $r_0$ for some $P_0,r_0$. It appears as a tree of $\mathbb{R}^2$. Let $\Gamma$ be the fundamental group of $Y$. It is clear that $X$ can be written as the union of the minimal sets of the non-trivial isometries of $\Gamma$, which is slim because it is torsion-free. However the conclusions of Theorem \ref{teor-dias} do not hold for $X$, showing that it cannot be decomposed as in Theorem \ref{maintheorem}. 
\end{ex}

 We would also bring to the attention of the reader a general observation. When $\Gamma$ is a discrete, torsionless group of isometries of our CAT(0)-space  $X$, then the quotient space $M=\Gamma \backslash X$ is a locally CAT$(0)$-space which is endowed with a natural measure $\mu_{M}$ (defined, locally, as explained in Section \ref{subsec-cat(0)}). If dias$(\Gamma,X) \geq \eta_0$,  then there is a ball of $M$ of radius $\frac{\eta_0}{2}$ which is isometric to a ball of $X$ and in particular (using the volume estimates recalled in Section \ref{subsec-packed,margulis}) we get $$\mu_{M}(M) \geq v_0 = v_0(P_0,r_0) > 0.$$
 Therefore, we can read all the above results   at level of quotient spaces as a ``small volume versus rigidity'' alternative. By the way of example,  we just express Theorem  \ref{teor-dias}(ii)  in this terms.
%\begin{cor}[Corollary \ref{corollary-CAT-epsilon} revisited]
%	Let $X$ be a complete, geodesically complete, \textup{CAT}$(0)$-space which is $P_0$-packed at scale $r_0$ and let $\Gamma$ be a discrete, torsion-free, cocompact group of isometries of $X$. Suppose there exists an open set $U$ of $X$ which is \textup{CAT}$(-\varepsilon)$ for some $\varepsilon > 0$. If $\mu_{\Gamma \backslash X}(\Gamma \backslash X) < v_0$ then $\Gamma \backslash X$ is isometric to $\mathbb{S}^1$ with a suitable normalized metric.
%\end{cor}
\begin{corollary}[Thm. \ref{teor-dias}(ii), revisited]
	Let $M$ be a compact, locally geodesically complete, locally \textup{CAT}$(0)$-space whose universal cover is $P_0$-packed at scale $r_0$. Suppose there exists an (arbitrarily small) open set $U\subset M$ which is \textup{CAT}$(-\varepsilon)$ for some $\varepsilon > 0$. If $\mu_{M}(M) < v_0$ then $M$, up to renormalization, is isometric to $\mathbb{S}^1$.
\end{corollary}

\section{Preliminaries}
\label{sec-preliminaries}

The open (resp. closed) ball of center $x$ and radius $r$ in a metric space $X$ will be denoted by $B(x,r)$ (resp. $\overline{B}(x,r)$). A geodesic in a metric space $X$ is an isometry $c\colon [a,b] \to X$, where $[a,b]$ is an interval of $\mathbb{R}$. The endpoints of the geodesic $c$ are the points $c(a)$ and $c(b)$; a geodesic with endpoints $x,y\in X$ is also denoted by $[x,y]$. A geodesic ray is an isometry $c\colon [0,+\infty) \to X$ and a geodesic line is an isometry $c\colon \mathbb{R} \to X$.
A metric space $X$ is called   {\em geodesic}  if for every two points $x,y \in X$ there is a geodesic with endpoints   $x$ and $y$.

\subsection{Geodesically complete CAT(0)-spaces}
 
\label{subsec-cat(0)}

A metric space $X$ is CAT$(0)$ if it is geodesic and every geodesic triangle $\Delta(x,y,z)$  is thinner than its Euclidean comparison triangle   $\overline{\Delta} (\bar{x},\bar{y},\bar{z})$, that is  for any couple of points $p\in [x,y]$ and $q\in [x,z]$ we have $d(p,q)\leq d(\bar{p},\bar{q})$ where $\bar{p},\bar{q}$ are the comparison points of $p,q$ in $\overline{\Delta} (\bar{x},\bar{y},\bar{z})$ (see   for instance \cite{BH09} for the basics of CAT(0)-geometry).\\
As consequence of the definition every CAT$(0)$-space is uniquely geodesic: for every two points $x,y$ there exists a unique geodesic segment with endpoints $x$ and $y$. A CAT$(0)$-metric space is {\em geodesically complete}   if any geodesic segment $c\colon [a,b] \to X$ can be extended to a geodesic line. For instance a complete CAT$(0)$-homology manifold is always geodesically complete (\cite{BH09}, Proposition II.5.12).\\
A subset $D$ of a complete, geodesically complete CAT$(0)$-space $X$ is {\em convex} for every two points $x,y\in D$ the geodesic $[x,y]$ is contained in $D$.\\
The space of direction $\Sigma_xX$ at a point $x\in X$ is the set of geodesic segments starting from $x$ modulo the equivalence relation $c\sim c'$ if and only if $\angle_x(c,c') = 0$. Given $x,y\in X$ we denote by $[y]_x$ the element of $\Sigma_xX$ defined by the geodesic segment $[x,y]$. \\
For a CAT$(0)$ space we denote by GD$(X)$ the {\emph{geometric dimension}} of $X$ as defined in \cite{Kl99}. By definition GD$(X) = 0$ if $X$ is discrete and 
$$\text{GD}(X) = 1+\sup_{x \in X}\text{GD}(\Sigma_x X).$$
In particular a CAT$(0)$-space has geometric dimension $1$ if and only if all its spaces of directions are discrete.
\vspace{2mm}

\noindent{\em In the following, we will always assume that $X$ is a proper, geodesically complete \textup{CAT}$(0)$-space.} By \cite{LN19} we know that every point $x\in X$ has a well defined integer dimension in the following sense: there exists $n_x\in \mathbb{N}$ such that every small enough ball around $x$ has Hausdorff dimension equal to $n_x$. This defines a stratification of $X$ into pieces of different integer dimensions. Namely, if $X^k$ denotes the subset of points of $X$ with dimension $k$, then $$X= \bigcup_{k\in \mathbb{N}} X^k.$$
The set $X_k$ contains a dense subset $M_k$, which is open in $X$ and locally bi-Lipschitz homeomorphic to $\mathbb{R}^k$ (\cite{LN19}, Theorem 1.2). In particular if $X_k \neq \emptyset$ then there exists some ball in $X$ which is a $k$-dimensional Lipschitz manifold.
The {\em dimension} of $X$ is the supremum of the dimensions of its points:  it coincides with the {\em topological dimension} of $X$, cp. Theorem 1.1 of \cite{LN19},  and with the geometric dimension of $X$.
The formula $$\mu_X := \sum_{k\in \mathbb{N}} \mathcal{H}^k \llcorner X^k$$  where $\mathcal{H}^k$ is the $k$-dimensional Hausdorff measure, defines a natural measure on $X$ which is locally positive and locally finite.
 
\subsection{Isometries of CAT(0)-spaces}

\label{subsec-isometries}
The group of isometries of $X$ will be denoted by Isom$(X)$. The translation length of $g\in \text{Isom}(X)$ is by definition $$\ell(g) := \inf_{x\in X}d(x,gx).$$
 If the infimum is realized then $g$ is called {\em elliptic} if $\ell(g) = 0$ and {\em hyperbolic} otherwise;   $g$ is called {\em parabolic} when the infimum is not realized. \\
 It is well known that the function $d_g \colon X \to [0,+\infty)$, $x\mapsto d(x,gx)$ is convex. Its $\tau$-level sets, for $\tau \geq 0$, are by definition the subsets 
 $$M_\tau(g) := \lbrace x\in X \text{ s.t. } d(x,gx) \leq \tau\rbrace$$  which are closed, convex subsets of $X$. We recall that the {\em minimum set of $g$} is defined as 
 $$\text{Min}(g):=M_{\ell(g)}(g)$$ that is, $\text{Min}(g)$ is the subset of points of $X$ realizing the minimum in the definition of the translation length of $g$. Notice that if $g$ is parabolic then $\text{Min}(g) = \emptyset$, while if $g$ is elliptic then $\text{Min}(g)$ is the set of points fixed by $g$, in which case it will be denoted also by $\text{Fix}(g)$. \\
 It is known that if $g$ in an hyperbolic isometry then $\text{Min}(g)$ splits isometrically as a product $D(g) \times \mathbb{R}$, where $D(g)$ is a convex subset of $X$, and $g$ acts on $D(g) \times \mathbb{R}$ respecting the product decomposition and acting as the identity on $D(g)$ and as a translation of length $\ell(g)$ on $\mathbb{R}$. \\
 When dealing with torsion there is a class of elliptic isometries which deserves a name. We say that an elliptic isometry $g$ of $X$ is \emph{slim} if $\text{Fix}(g)$ has empty interior. A group of isometries $\Gamma$ is said to be \emph{slim} if every non-trivial elliptic isometry of $\Gamma$ is slim. For instance every torsion-free group is trivially slim, as well as any discrete group of isometries of a CAT$(0)$-homology manifold. 
 \begin{lemma}
 	\label{lemma-class-manifolds}
 	Let $X$ be a proper \textup{CAT}$(0)$-space which is a homology manifold, and let $\Gamma$ be a discrete group of isometries of $X$. Then every elliptic isometry of $\Gamma$ is slim.
 \end{lemma}
 \begin{proof}
 	By \cite{LN-finale-18}, Theorem 1.2 there exists a locally finite subset $E$ of $X$ such that $X\setminus E$ is a topological manifold.  The set $E$ can clearly be chosen to be invariant under the isometry group of $X$. We can assume also that the dimension of $X$ is $\geq 3$, otherwise $X$ is already a topological manifold. In this case the topological manifold $X\setminus E$ is connected. Let $g$ be an elliptic isometry of $\Gamma$. Since $\Gamma$ is discrete,   the order of $g$ is finite. After taking a power of $g$, whose fixed point set contains the fixed point set of $g$, we can suppose $g$ has prime order. By Newman's second theorem \cite{New31} if $g$ fixes an open subset of $X\setminus E$ then it is the identity. This shows that Fix$(g) \cap (X\setminus E)$ has empty interior. Also $E$ has empty interior, so Fix$(g)$ has empty interior.
 \end{proof}
Let now $\Gamma$ be a subgroup of Isom$(X)$. For $x\in X$ and $\eta\geq 0$ we set 
$$\Sigma_\eta(x) := \lbrace g\in \Gamma \text{ s.t. } d(x,gx) < \eta\rbrace$$
$$\overline{\Sigma}_\eta(x) := \lbrace g\in \Gamma \text{ s.t. } d(x,gx) \leq \eta\rbrace$$ 
and we define $\Gamma_\eta(x) := \langle \Sigma_\eta(x) \rangle$, $\overline{\Gamma}_\eta(x) = \langle \overline{\Sigma}_\eta(x) \rangle$. For instance when $\eta=0$ we have $\overline{\Sigma}_0(x) = \overline{\Gamma}_0(x) = \text{Stab}_\Gamma(x)$, the stabilizer of $x$.
The subgroup $\Gamma$ is discrete if for all $x\in X$ and all $\eta\geq 0$ the set $\Sigma_\eta(x)$ is finite. \\
The {\em systole} of $\Gamma$ {\em at a point} $x\in X$ is 
$$\text{sys}(\Gamma,x) := \inf_{g\in \Gamma^*} d(x,gx)$$
%, \qquad \text{sys}^\diamond(\Gamma,x) := \inf_{g\in \Gamma\setminus \Gamma^\diamond} d(x,gx),$$
where $\Gamma^* = \Gamma \setminus \lbrace \text{id} \rbrace$.
% and $\Gamma^\diamond = \lbrace g \in \Gamma \text{ s.t. } g \text{ is elliptic}\rbrace$. \\ The  {\em (global) systole} and the  {\em free-systole of} $\Gamma$ are accordingly defined as 
%$$\text{sys}(\Gamma,X) = \inf_{x\in X}\text{sys}(\Gamma,x), \qquad \text{sys}^\diamond(\Gamma,X) = \inf_{x\in X}\text{sys}^\diamond(\Gamma,x).$$
%{\color{blue} ma la usiamo mai la free systole?}\\
Finally, the central notion of this paper is  the {\em diastole} of $\Gamma$, that is the quantity:
$$\text{dias}(\Gamma,X) = \sup_{x\in X}\text{sys}(\Gamma,x).$$

\noindent We end this section with some facts about finite sets of isometries we will need later. The first  one is a well-known fact about commuting isometries.

\begin{prop}[\cite{Gel11}, Lemma 2.7]
	\label{prop-commuting-isometries}
	Let $X$ be a complete \textup{CAT}$(0)$-space and let $g_1,\ldots,g_n$ be commuting isometries of $X$. Let $\tau_1,\ldots,\tau_n \geq 0$ such that $M_{\tau_i}(g_i) \neq \emptyset$. Then $\bigcap_{i=1}^n M_{\tau_i}(g_i) \neq \emptyset$.
\end{prop}
\noindent The second one is a similar statement for elliptic isometries:
\begin{prop}[\cite{BH09}, Corollary II.2.8]
	\label{prop-fix-points-finite}
	Let $X$ be a complete \textup{CAT}$(0)$-space and let $F$ be a finite group of isometries of $X$. Then $\bigcap_{g\in F}  \textup{Fix}(g) \neq \emptyset$.
\end{prop}
 
 \subsection{Packing conditions on CAT(0)-spaces}

\label{subsec-packed,margulis}
Let $P_0,r_0 > 0$: we say $X$ is {\em $P_0$-packed at scale $r_0$} if for every $x\in X$ the cardinality of every $r_0$-separated subset of $\overline{B}(x,3r_0)$ is at most $P_0$ (recall that a subset $Y\subseteq X$ is {\em $r_0$-separated} subset if $d(y,y') > r_0$ for all $y,y'\in Y$). \linebreak We will simply say that {\em $X$ is packed} if it is $P_0$-packed at scale $r_0$ for some $P_0,r_0>0$. If $X$ is a complete, geodesically complete CAT$(0)$-space which is $P_0$-packed at scale $r_0$ then the dimension of $X$ is at most $\frac{P_0}{2}$, and for every $R\geq 0$ and for every $x\in X$ it holds:
\begin{equation}
	\label{eq-volume-estimate}
	v(R) \leq \mu_X(B(x,R)) \leq V(R),
\end{equation}
where $v(R)$ and $V(R)$ are functions  only depending on the packing constants  $P_0,r_0$, as showed in  \cite{CavS20}.
Another remarkable consequence of a  packing condition at some fixed scale   is 
the following version of the Margulis' Lemma due to Breuillard-Green-Tao, which we decline for geodesically complete CAT(0)-spaces.

%an analogue for metric spaces of the celebrated {\em Margulis's Lemma} for Riemannian manifolds: 
%{\em there exists a constant $\varepsilon_0$, only depending on the packing constants $P_0,r_0$,
%%(only depending on the $n$-dimensional macroscopic scalr curvature bound) 
%such that for every discrete group of isometries $\Gamma$ of   $X$ 
%%which is  $P_0$-packed at scale $r_0$, 
%the 	$\varepsilon_0$-almost stabilizer $\Gamma_{\varepsilon_0} (x)$ of any point $x$ is virtually nilpotent}  (cp. \cite{BGT11}, Corollary 11.17); that is, the elements of $\Gamma$ which displace $x$ less than $\varepsilon_0$ generate a virtually nilpotent group. \\one is the following version of the Margulis lemma.  
\begin{theo}[\cite{BGT11}, Corollary 11.17 and Corollary 11.2]
	\label{theorem-Margulis-nilpotent}
	Let $X$ be a complete, geodesically complete, \textup{CAT}$(0)$-space which is $P_0$-packed at scale $r_0$ and $\Gamma$ be a discrete group of isometries of $X$. There exist $\varepsilon_0 = \varepsilon_0(P_0,r_0) > 0,$ $I_0 = I_0(P_0,r_0) \geq 0$ and $S_0 = S_0(P_0,r_0) \geq 0$ such that the following holds true: for every $x\in X$ and  every positive $\eta \leq \varepsilon_0$ there exist normal subgroups $\Delta_x, \Lambda_x$ of $\overline{\Gamma}_{\eta}(x)$ such that
	\begin{itemize}
		\item[(i)] $[\overline{\Gamma}_{\eta}(x) : \Delta_x] \leq I_0$;
		\item[(ii)] $\Lambda_x$ is a finite subgroup of $\Delta_x$;
		\item[(iii)] $\Delta_x / \Lambda_x$ is nilpotent of step $\leq S_0$.
	\end{itemize}
	In particular $\overline{\Gamma}_{\eta}(x)$ is virtually nilpotent. Moreover, if $\Gamma$ is cocompact then   $\overline{\Gamma}_{\eta}(x)$ is virtually abelian (cp. \cite{BH09}, Theorem II.7.8).
\end{theo}
%\begin{obs}
%	Under the assumptions of Theorem \ref{theorem-Margulis-nilpotent}, if $\Gamma$ is cocompact then we can conclude that $\overline{\Gamma}_{\eta}(x)$ is virtually abelian by \cite{BH09}, Theorem II.7.8.
%\end{obs}
 
We call this $\varepsilon_0(P_0, r_0)$ the {\em Margulis constant}, since it plays the role of the classical Margulis constant in this metric setting. This constant and Theorem \ref{theorem-Margulis-nilpotent} will play a crucial role in our main theorem.\\
We end this section describing more carefully the structure of the groups $\overline{\Gamma}_\eta(x)$ appearing in Theorem \ref{theorem-Margulis-nilpotent}.

	\begin{lemma}[\cite{BGS13}, Lemma 7.4 and Remark 7.2]
		\label{lemma-BGS-center}
		Let $X$ be a complete, \textup{CAT}$(0)$-space and let $N$ be a nilpotent group of isometries of $X$. If $N$ has a semisimple isometry then there exists a semisimple isometry in the center of $N$. The set of semisimple isometries in the center of $N$ forms an abelian normal subgroup of $N$.
	\end{lemma}
	\begin{prop}
		\label{prop-structure-virt-nilpotent}
		Let $X$ be a complete, \textup{CAT}$(0)$-space with $\textup{GD}(X) < + \infty$. Let $\Gamma$ be a finitely generated, discrete, virtually nilpotent group of isometries of $X$. Then:
		\begin{itemize}
			\item[(i)]  there exists a $\Gamma$-invariant convex subset $W \times \mathbb{R}^k$, $k\geq 0$, of $X$ such that each $\gamma \in \Gamma$ acts as $(\gamma', \gamma'')$, where $\gamma'$ is not hyperbolic;
			\item[(ii)] a finite index, normal subgroup of $\Gamma$ splits as $\Gamma_W \times \mathbb{Z}^k$, where $\Gamma_W$ is a finitely generated, discrete, nilpotent group acting on $W$ without hyperbolic isometries and $\mathbb{Z}^k$ acts as a lattice on $\mathbb{R}^k$;
			\item[(iii)] the set of semisimple isometries of $\Gamma_W \times \mathbb{Z}^k$ is a virtually abelian, normal subgroup of $\Gamma$. It coincides with the set of elements whose first component is elliptic;
			\item[(iv)] if $\Gamma$ does not contain parabolic elements then it is virtually abelian.
		\end{itemize} 
	\end{prop}
	\begin{proof}
		We show (i) and (ii) by induction on $n:=\text{GD}(X)$, where the case $n=0$ is trivial. If $\Gamma$ does not contain hyperbolic isometries it is enough to take $W=X$ and $k=0$. Suppose $\Gamma$ has a hyperbolic isometry and denote by $N$ a nilpotent subgroup of finite index. Since $N$ is finitely generated and nilpotent it is not restrictive to suppose $N$ torsion-free and normal in $\Gamma$. Clearly also $N$ has hyperbolic isometries, so there exists a hyperbolic isometry in the center of $N$ by Lemma \ref{lemma-BGS-center}. We consider the set $H$ of semisimple isometries in the center of $N$: it is an abelian subgroup of $N$ by Lemma \ref{lemma-BGS-center}. Let $\gamma_1,\ldots,\gamma_n \in \Gamma$ such that $\Gamma = \bigcup_{i=1}^n\gamma_iN$. Observe that the isometries of $\gamma_i H \gamma_i^{-1}$ are still semisimple and belong to the center of $N$. Indeed the center of $N$ is characteristic in $N$, so it is normal in $\Gamma$.
		The isometries $\lbrace \gamma_i H \gamma_i^{-1} \rbrace_{i=1,\ldots,n}$ generate an abelian group which is clearly normal in $\Gamma$ and has positive rank. Up to passing to the characteristic subgroup of this abelian group we can suppose to have an abelian, torsion-free, normal subgroup $A$ of $\Gamma$. By the flat torus Theorem (cp. \cite{BH09}, Theorem II.7.1) we know that there exists a convex subset $W' \times \mathbb{R}^j$ of $X$ which is preserved by $\Gamma$, and moreover each $\gamma \in \Gamma$ acts as $(\gamma',\gamma'')$. Furthermore $\Gamma$ has a subgroup of finite index which splits as $\Gamma_{W'} \times A \cong \Gamma_{W'} \times \mathbb{Z}^j$. Let $p(\Gamma)$ be the projection of $\Gamma$ on $\text{Isom}(W')$. Clearly it is virtually nilpotent and finitely generated. We need to show it is discrete. It contains $\Gamma_{W'}$ as finite index subgroup and it is easy to see that $\Gamma_{W'}$ is discrete, so also $p(\Gamma)$ is discrete. Finally observe that since the rank of $A$ is strictly positive we have $\text{GD}(W') < \text{GD}(X)$, so we can apply induction on $p(\Gamma)$. So there exists a convex subset $W\times \mathbb{R}^k$ which is $p(\Gamma)$-invariant and each isometry of $p(\Gamma)$ preserves the product decomposition, i.e. each $\gamma \in p(\Gamma)$ decomposes as $(\gamma',\gamma'')$ with $\gamma'$ not hyperbolic. This means that the set $W\times \mathbb{R}^k \times \mathbb{R}^j$ does the job for $\Gamma$. Moreover by induction we can also assume that a finite index subgroup of $\Gamma_{W'}$ splits as $\Gamma_W \times \mathbb{Z}^k$, so a finite index subgroup of $\Gamma$ splits as $\Gamma_W \times \mathbb{Z}^{k+j}$. \\
		An isometry $(\gamma',\gamma'')$ of $\Gamma_W \times \mathbb{Z}^k$ is semisimple if and only if both $\gamma'$ and $\gamma''$ are semisimple (\cite{BH09}, Proposition II.6.9). Since $\gamma''$ is always semisimple because it is an isometry of $\mathbb{R}^k$, we conclude that $(\gamma',\gamma'') \in \Gamma_W \times \mathbb{Z}^k$ is semisimple if and only if $\gamma'$ has finite order. Since the finite order elements of a nilpotent group form a subgroup, we see that the set of semisimple isometries of $\Gamma_W \times \mathbb{Z}^k$ are a subgroup. Let $(\gamma',\gamma'')$ be a semisimple isometry of $\Gamma_W \times \mathbb{Z}^k$ and $(g',g'')$ be an element of $\Gamma$. Then $$(g',g'')(\gamma',\gamma'')(g'^{-1}, g''^{-1}) = (g'\gamma'g'^{-1}, g''\gamma''g''^{-1}).$$
		Since $\Gamma_W \times \mathbb{Z}^k$ is normal in $\Gamma$ we know that $(g'\gamma'g'^{-1}, g''\gamma''g''^{-1})$ still belongs to $\Gamma_W \times \mathbb{Z}^k$. Moreover $g'\gamma'g'^{-1}$ is still of finite order: this shows that the subgroup of semisimple isometries of $\Gamma_W \times \mathbb{Z}^k$ is normal in $\Gamma$. Clearly the quotient of this group by the torsion group of $\Gamma_W$, which is finite, is abelian, so this group is virtually abelian.\\
		Finally (iv) follows immediately: indeed if $\Gamma$ does not have parabolic elements then also $\Gamma_W \times \mathbb{Z}^k$ does not have parabolic elements. This means that $\Gamma_W$ is finite, so $\Gamma_W \times \mathbb{Z}^k$ is virtually abelian. Therefore also $\Gamma$ is virtually abelian.
%		
%		
%		
%		By the flat torus Theorem (cp. \cite{BH09}, Theorem II.7.1.(5)) there exists a finite index subgroup $N'$ of $N''$ of the form $N'_D \times \langle g \rangle$ acting on $\text{Min}(g) = D(g)\times \mathbb{R}$. The group $N'_D$ is nilpotent, discrete and finitely generated, because $N'$ is, and $\text{GD}(D(g)) < \text{GD}(X)$. By induction we deduce that there exists a convex subset of $D(g)$ of the form $W\times \mathbb{R}^k$ and a finite index subgroup $N_W' \times \mathbb{Z}^k$ of $N_D'$ such that each isometry of $N_W'$ is either elliptic or hyperbolic and $\mathbb{Z}^k$ acts as a lattice on $\mathbb{R}^k$. The group $N_W' \times \mathbb{Z}^k \times \langle g \rangle \cong N_W'\times \mathbb{Z}^{k+1}$ has finite index in $N$ and acts on $W\times \mathbb{R}^{k+1}$ with the required properties. 
%		
%			
	\end{proof}

\subsection{Almost cocompact  groups of CAT(0)-spaces}

\label{subsec-almost cocompact}
A  group of isometries $\Gamma$ of    a complete, geodesically complete CAT$(0)$-space $X$ is called {\em cocompact} if there exists a compact subset $K \subset X$ such that $X = \bigcup_{g\in \Gamma} gK$. The main theorem of the paper will be proved 
%for discrete, cocompact group of isometries. However it can be proved 
under a weaker assumption, which we call  {\em almost-cocompactness}, which   we now describe.\\ 
 For $\delta > 0$ we define the {\em $\delta$-thick subset} of $X$ as 
$$\text{Thick}_\delta(\Gamma, X):=\lbrace x\in X \text{ s.t. sys}(\Gamma,x)\geq \delta\rbrace.$$
This is clearly a $\Gamma$-invariant subset of $X$. We say $\Gamma$ is {\em almost-cocompact} if the action of $\Gamma$ on $\text{Thick}_\delta(\Gamma,X)$ is cocompact for all $\delta > 0$. In \cite{BGS13}, Section 8, the same notion for manifolds was introduced under the notation InjRad$(\Gamma\backslash X) \to 0$.\\
It is clear that if $\Gamma$ is cocompact then it is almost-cocompact. Moreover if the global systole of the action is strictly positive then almost-cocompactness coincides with cocompactness. Among almost-cocompact actions the most important ones are those with finite covolume:
\begin{defin}
	Let $\Gamma$ be a discrete subgroup of isometries of a complete, geodesically complete CAT$(0)$-space  $X$. An open subset $\mathcal{D} \subset X$ is a {\em fundamental domain} for $\Gamma$ if :
	\begin{itemize}
		\item[(i)] $g\mathcal{D} \cap h\mathcal{D} =\emptyset$ for all $g,h\in \Gamma$, $g \neq h$;
		\item[(ii)] $\bigcup_{g\in \Gamma} g\overline{\mathcal{D}} = X$.
	\end{itemize}
	The group $\Gamma$ is said to have {\em finite covolume} if there exists a fundamental domain $\mathcal{D}$ for $\Gamma$ with $\mu_X(\mathcal{D}) < + \infty$.
\end{defin}
\begin{lemma}[Compare with \cite{BGS13}, §8.4]
	\label{lemma-basic-thick}
	Let $X$ be a complete, geodesically complete, packed, \textup{CAT}$(0)$ space $X$ and let $\Gamma$ be a discrete group of isometries of $X$ with finite covolume. Then $\Gamma$ is almost-cocompact.
\end{lemma}
\begin{proof}
	Suppose $\Gamma\backslash\textup{Thick}_\delta(\Gamma, X)$ is not compact for some $\delta > 0$. It is complete, so it is not totally bounded. This means we can find infinitely many disjoint balls of radius $0<\varepsilon \leq \frac{\delta}{2}$ centered in points of $\Gamma \backslash \textup{Thick}_\delta(\Gamma, X)$. By definition these balls are isometric to the corresponding balls in $X$, in particular their volume is bigger than a universal number $v > 0$ depending on $\varepsilon$ and the packing constants of $X$, by \eqref{eq-volume-estimate}. This contradicts the finite covolume assumption.
\end{proof}

\noindent The existence of a fundamental domain is discussed in the following result.
\begin{prop}
	\label{prop-existence-fundamental-domain}
    Let $X$ be a complete \textup{CAT}$(0)$-space and let $\Gamma$ be a discrete group of isometries. Then there exists a fundamental domain for $\Gamma$ if and only if \textup{dias}$(\Gamma,X) > 0$, i.e. if there exists $x\in X$ such that $\textup{Stab}_\Gamma(x) = \lbrace \textup{id} \rbrace$. In particular fundamental domains exist when $\Gamma$ is slim, e.g. when it is torsion-free or $X$ is a homology manifold by Lemma \ref{lemma-class-manifolds}.
\end{prop}

\begin{proof}
	Clearly \textup{dias}$(\Gamma,X) > 0$ if and only if there exists $x\in X$ such that $\textup{Stab}_\Gamma(x) = \lbrace \textup{id} \rbrace$. If this is the case then the set 
	$$\mathcal{D} = \lbrace y \in X \text{ s.t. } d(y,x) < d(gy,x) \text{ for all } g \in \Gamma \rbrace$$ 
	is a fundamental domain, since $X$ is a geodesic space. 
	This assumption is satisfied if $\Gamma$ is slim: indeed by Baire theorem $\bigcup_{g\in \Gamma^*}\text{Fix}(g)$ has empty interior, because $\Gamma$ is countable. So there exists a point whose stabilizer is trivial.\\
	Viceversa let us suppose that $\textup{Stab}_\Gamma(x) \neq \lbrace \textup{id} \rbrace$ for every $x\in X$.
	If $\mathcal{D}$ is any subset of $X$ and $x \in \mathcal{D}$ then there exists a non-trivial $g\in \Gamma$ such that $gx=x$, so $g\mathcal{D} \cap \mathcal{D} \neq \emptyset$. Therefore fundamental domains cannot exist.	
\end{proof}

\noindent Example \ref{example-Lytchak} shows that there exist cocompact, discrete groups of isometries $\Gamma$ of complete, geodesically complete, CAT$(0)$-spaces $X$ with $\text{dias}(\Gamma,X) = 0$.

\noindent We conclude with an important property of almost-cocompact, slim groups. A group of isometries $\Gamma$ of a CAT$(0)$-space $X$ is said to be \emph{minimal} if there are no convex, closed, $\Gamma$-invariant proper subsets $C\subsetneq X$.
\begin{prop}[Compare with \cite{BGS13}, proof of Lemma 2 at page 195]
	\label{prop-minimal}
	Let $X$ be a proper, geodesically complete, \textup{CAT}$(0)$-space. Let $\Gamma$ be a discrete, slim, almost-cocompact group of isometries of $X$. Then $\Gamma$ is minimal.
\end{prop}
\begin{proof}
	Suppose to have a closed, convex, $\Gamma$-invariant subset $C\subsetneq X$. Let $x\in X\setminus C$ and $y\in C$ be its projection on $C$. Consider the geodesic ray $[y,x]$ and extend it to a geodesic ray $c_x$. By $\Gamma$-invariance of $C$ we know that $d_g$ is non-decreasing along $c_x$ for all $g\in \Gamma$. Suppose first there exist $\delta > 0$ and a time $t_0 \geq 0$ such that $d(c_x(t_0), gc_x(t_0)) \geq \delta > 0$ for all $g\in \Gamma^*$. The same holds for all $t\geq t_0$, so there exists a whole geodesic ray which is contained in $\text{Thick}_\delta(\Gamma,X)$. Clearly this ray does not project to a compact subset of $\Gamma \backslash X$, a contradiction to the almost-cocompactness assumption. In the remaining case we have that the ray $c_x$ is entirely contained in $\bigcup_{g\in \Gamma^*}\text{Fix}(g)$, and this must happen for all possible choices of $x\in X\setminus C$. In other words $X\setminus C \subseteq \bigcup_{g\in \Gamma^*}\text{Fix}(g)$. Now, $X\setminus C$ is an open, non-empty subset of $X$ while each $\text{Fix}(g)$ is a closed subset with empty interior. The union is countable because $\Gamma$ is discrete, so by Baire Theorem $\bigcup_{g\in \Gamma^*}\text{Fix}(g)$ has empty interior, impossible.
\end{proof}

\section{Thin actions by almost-cocompact groups}
\label{sec:alternative}

From now on we fix $P_0,r_0 > 0$ and we call $\varepsilon_0,I_0,S_0$ the constants given by Theorem \ref{theorem-Margulis-nilpotent}.
For $\lambda \geq 0 $ we introduce the set $$\text{Min}_\lambda(\Gamma,X) : = \bigcup_{g\in \Gamma^* \textup{ : } \ell(g)\leq \lambda} \textup{Min}(g).$$
In particular when $\lambda = 0$ we denote the set $\text{Min}_\lambda(\Gamma,X)$ by $\text{Fix}(\Gamma,X)$, which is the set of points that are fixed by some non-trivial (and necessarely elliptic) isometry.
Observe that $\text{Min}_\lambda(\Gamma,X)$ contains $\text{Fix}(\Gamma,X)$ for every $\lambda \geq 0$. \\
We state here the version of the main theorem we will prove.
\begin{theo}
	\label{theo-alternative}
	Let $X$ be a complete, geodesically complete, \textup{CAT}$(0)$-space that is $P_0$-packed at scale $r_0$. Let $\Gamma$ be a discrete, almost-cocompact group of isometries of $X$ and let $0<\lambda \leq \varepsilon_0$. Assume that \textup{dias}$(\Gamma,X) <  \eta(P_0,r_0,\lambda) := \frac{\lambda}{(2I_0-1)(3\cdot 2^{S_0-1} - 2)}$, then $X=\textup{Min}_\lambda(\Gamma,X)$. 
\end{theo}

% The following is essentially a rephrasing of Proposition \ref{prop-alternative}. 
%{\color{olive} E QUINDI TOGLIEREI ANCHE QUESTO}
%\begin{cor}\label{cor-prop-alternative}
%	Let $X$ be a complete, geodesically complete, \textup{CAT}$(0)$-space which is $P_0$-packed at scale $r_0$ and let $\Gamma$ be a discrete, cocompact group of isometries of $X$. Then one of the following two statements hold:
%	\begin{itemize}
%		\item[(a)] $\textup{dias}(\Gamma,X) \geq \eta_0$;
%		\item[(b)] $X$ admits a virtually abelian local splitting structure.
%	\end{itemize}
%\end{cor}
%
%
%\begin{proof}[Proof of Corollary \ref{cor-prop-alternative}]
%	By Proposition \ref{prop-alternative} if (a) does not hold then $$X = \bigcup_{g\in \Gamma^* \textup{ : } \ell(g)\leq \lambda_0} \text{Min}(g).$$ We claim that $\mathcal{A} = \lbrace (g,\ell(g)) \text{ s.t. } g\in \Gamma^*, \ell(g)\leq \lambda_0 \rbrace$ is a virtually abelian structure. Conditions (a) and (b) are trivially satisfied. Suppose that Min$(g)\cap\text{Min}(g') \neq \emptyset$ and take a point $x$ in the intersection. Both $g,g'$ belong to $\overline{\Sigma}_{\lambda_0}(x) \subseteq \overline{\Sigma}_{\varepsilon_0}(x)$, so $\langle g,g' \rangle$ is contained in $\overline{\Gamma}_{\varepsilon_0}(x)$. This group is virtually abelian by Theorem \ref{theorem-Margulis-nilpotent}. It is clear that $\langle g, g' \rangle$ is virtually abelian. Indeed let $G$ be an abelian subgroup of $\overline{\Gamma}_{\varepsilon_0}(x)$ of finite index. Then $\langle g, g' \rangle \cap G$ is an abelian subgroup of finite index in $\langle g, g' \rangle$.
%\end{proof}	

Example \ref{example-Lytchak} in which dias$(\Gamma,X) = 0$ shows that if case (b) occurs the elements of the decomposition can be all elliptic. Notice that if in the above decomposition Min$(g)$ intersects Min$(g')$ then $\lbrace g,g' \rbrace$ generate a virtually nilpotent group (virtually abelian, if $\Gamma$ is cocompact), by Theorem \ref{theorem-Margulis-nilpotent}.\\  
To prove Theorem \ref{theo-alternative} we need a couple of preparatory  results. The first one is a group-theoretic result for finite index subgroups:  

\begin{lemma}[Lemma 3.4 of \cite{SW}]
	\label{lemma-finite-index-generators}
	Let $G$ be a group generated by a finite set $S$. Let $H$ be a subgroup of $G$ with finite index $I$. Then $H$ can be generated by a set $S'$ whose elements have   word length with respect to $S$ at most $(2I-1)$.
\end{lemma}

\vspace{-2mm}
\noindent (We recall that, in the above setting,  the  word length with respect to $S$ of an element $g\in G$ is the infimum of the natural numbers $n$ such that $g$ can be written as a product of $n$ elements of $S$.)

\begin{lemma}
	\label{lemma-basic-containment}
	For every $x\in X$ and $r > 0$ there exists an open set $U \ni x$ such that $\Sigma_\eta(y)\supseteq \Sigma_\eta(x)$ and $\overline{\Sigma}_\eta(y)\subseteq \overline{\Sigma}_\eta(x)$ for all $y\in U$.
\end{lemma}
\begin{proof}
	By discreteness $\Sigma_\eta(x) = \lbrace g_1,\ldots, g_k\rbrace$ and $d(x,g_ix)<\eta$ for each $i=1,\ldots, k$. Hence there exists $\varepsilon > 0$ such that $d(x,g_ix)<\eta - 2\varepsilon$ for every $i = 1,\ldots, k$. So we have $\Sigma_\eta(y)\supseteq \Sigma_\eta(x)$ for every point $y\in B(x,\varepsilon)$.\\
	In order to prove the second part we suppose the opposite: there is a sequence $x_n$ converging to $x$ such that $\overline{\Sigma}_\eta(x_n)\supset \overline{\Sigma}_\eta(x)$ for every $n$: in particular there is $g_n \in \Gamma$ such that $d(x_n,g_nx_n)\leq \eta$ but $d(x,g_nx)>\eta$. By discreteness the sets of possible $g_n$'s is finite, so we can take a subsequence where $g_n$ is constantly equal to a fixed $g\in \Gamma$. Therefore by continuity $d(x,gx) = \lim_{n\to +\infty} d(x_n, g x_n) \leq \eta$, which is a contradiction.
\end{proof}

We are ready to prove Theorem \ref{theo-alternative}. We present it in the cocompact case and we will see after the proof how to do the general case.
\begin{proof}[Proof of Theorem \ref{theo-alternative}]
	We suppose $X\neq \text{Min}_\lambda(\Gamma,X)$. Our aim is to find a point $x\in X$ such that sys$(\Gamma,x)\geq \eta = \eta(P_0,r_0,\lambda)$.
	Let $f\colon [0,+\infty) \to [0,+\infty)$ be a continuous, decreasing function such that $\lim_{t\to 0}f(t) = +\infty$ and $f(t)=0$ for all $t\geq \eta$. We define the map $\Psi \colon X\setminus \text{Min}_\lambda(\Gamma,X) \to [0,+\infty)$ by
	\begin{equation}
		\label{Psi}
		\Psi(x) = \sum_{g\in \Sigma_{\eta}(x) \setminus \lbrace \text{id} \rbrace}f(d(x,gx)) = \sum_{g\in \overline{\Sigma}_{\eta}(x)\setminus \lbrace \text{id} \rbrace}f(d(x,gx)),
	\end{equation}
	where the sum is equal to $0$ if $\Sigma_{\eta}(x) = \lbrace \text{id} \rbrace$.
	This map is well defined on every point $x$ which does not belong to $\text{Fix}(\Gamma,X)$, in particular it is well defined outside $\text{Min}_\lambda(\Gamma,X)$. In order to get the thesis it is enough to show that $\min\Psi = 0$. Indeed if a point $x\in X\setminus \text{Min}_\lambda(\Gamma,X)$ satisfies $\Psi(x)=0$ then sys$(\Gamma,x)\geq \eta$.
	\vspace{3mm}
	
	\noindent \textbf{Step 1: $\Psi$ is continuous and $\Gamma$-invariant.} \\
	The $\Gamma$-invariance is obvious. Let us take a sequence $x_n$ converging to $x_\infty$. By Lemma \ref{lemma-basic-containment} we have $\Sigma_{\eta}(x_n) \supseteq \Sigma_{\eta}(x_\infty)$ and $\overline{\Sigma}_{\eta}(x_n) \subseteq \overline{\Sigma}_{\eta}(x_\infty)$, for $n$ big enough. Therefore, using the two expressions of $\Psi$ in \eqref{Psi}:
	\begin{equation*}
		\begin{aligned}
			\liminf_{n\to + \infty} \Psi(x_n) &= \liminf_{n\to + \infty} \sum_{g\in \Sigma_{\eta}(x_n) \setminus \lbrace \text{id} \rbrace}f(d(x_n,gx_n))\\ &\geq \liminf_{n\to + \infty} \sum_{g\in \Sigma_{\eta}(x_\infty) \setminus \lbrace \text{id} \rbrace}f(d(x_n,gx_n)) \\
			&= \sum_{g\in \Sigma_{\eta}(x_\infty) \setminus \lbrace \text{id} \rbrace}f(d(x_\infty,gx_\infty)) = \Psi(x_\infty)
		\end{aligned}
	\end{equation*}
	and
	\begin{equation*}
		\begin{aligned}
			\limsup_{n\to + \infty} \Psi(x_n) &= \limsup_{n\to + \infty} \sum_{g\in \overline{\Sigma}_{\eta}(x_n) \setminus \lbrace \text{id} \rbrace}f(d(x_n,gx_n)) \\
			&\leq \limsup_{n\to + \infty} \sum_{g\in \overline{\Sigma}_{\eta}(x_\infty) \setminus \lbrace \text{id} \rbrace}f(d(x_n,gx_n)) \\
			&= \sum_{g\in \overline{\Sigma}_{\eta}(x_\infty) \setminus \lbrace \text{id} \rbrace}f(d(x_\infty,gx_\infty)) = \Psi(x_\infty).
		\end{aligned}
	\end{equation*}\\
		\textbf{Step 2: $\Psi$ has a minimum.} \\
		Let us take a sequence $x_n \in X\setminus \text{Min}_\lambda(\Gamma,X)$ such that $\Psi(x_n) \to \inf \Psi$. By cocompactness of the action we can suppose that $x_n$ converges to $x_\infty$ and, since $\Psi$ is continuous, $\Psi(x_\infty) = \inf \Psi$.
	\vspace{2mm}
	
	In the next step we will find a good point realizing the minimum of $\Psi$. In order to do that we introduce the partial order $\preceq$ on $\Psi^{-1}(m)$, $m\in [0,+\infty)$, defined by $x \preceq y$ if and only if $\overline{\Sigma}_{\eta}(x) \subseteq \overline{\Sigma}_{\eta}(y)$.
	\vspace{3mm}
	
	\noindent \textbf{Step 3: the mininum of $\Psi$ is realized at a point $x$  maximal for $\preceq$.} \\
	Set $m := \min\Psi$. We consider the set $\Psi^{-1}(m)$ with the partial order $\preceq$. We claim that every chain has a maximal element. Let us fix a chain $x_1 \preceq x_2 \preceq \ldots$. The function $x\mapsto \#\overline{\Sigma}_\eta(x)$ is upper semicontinuous by Lemma \ref{lemma-basic-containment} and $\Gamma$-invariant, so it has a maximum on $X$. Moreover $\#\overline{\Sigma}(x_1) \leq \#\overline{\Sigma}(x_2) \leq \ldots$, so this sequence must stabilize, i.e. there exists $\bar{n}\in \mathbb{N}$ such that $\overline{\Sigma}(x_{\bar{n}}) = \overline{\Sigma}(x_n)$ for all $n\geq \bar{n}$. This means that $x_{\bar{n}}$ is a maximizing element of the chain. By Zorn's Lemma there exists a point $x$ such that $\Psi(x) = m$ and which is maximal for $\preceq$.
	\vspace{2mm}
	
	Our aim is to show that $m = 0$, so we suppose $m > 0$ and we look for a contradiction. Indeed we will find a point $y\in X\setminus \text{Min}_\lambda(\Gamma,X)$ such that $\Psi(y) < m$. Since $m > 0$ the set $\overline{\Sigma}_{\eta}(x)$ contains some non-trivial element. We write $\overline{\Sigma}_{\eta}(x) = \lbrace \sigma_1,\ldots,\sigma_n\rbrace$ and we consider the group $\overline{\Gamma}_\eta(x) = \langle \overline{\Sigma}_{\eta}(x)\rangle$. 
	\vspace{-1mm}
	
	\noindent \textbf{Step 4: there exists a non-empty, closed, convex, $\overline{\Gamma}_\eta(x)$-invariant subset $Y$ of $X$ not containing $x$.} \\
	By Theorem \ref{theorem-Margulis-nilpotent} there are groups $\Lambda_x, \Delta_x \vartriangleleft \overline{\Gamma}_\eta(x)$ such that:
	\begin{itemize}
		\item[(i)] $[\overline{\Gamma}_\eta(x) : \Delta_x] \leq I_0$;
		\item[(ii)] $\Lambda_x$ is a finite subgroup of $\Delta_x$;
		\item[(iii)] $\Delta_x / \Lambda_x$ is nilpotent of step $\leq  S_0$. 
	\end{itemize}
	Here we divide the proof in two cases: (a) $\Lambda_x \neq \lbrace \text{id}\rbrace$ and (b) $\Lambda_x = \text{id}$ and hence $\Delta_x$ is nilpotent of step $\leq S_0$. \\
	In case (a) we define the set $Y = \bigcap_{g\in \Lambda_x}\text{Min}(g)$. It is closed, convex and non-empty by Proposition \ref{prop-fix-points-finite} because $\Lambda_x$ is finite. Moreover it is $\overline{\Gamma}_\eta(x)$-invariant since $\Lambda_x$ is normal in $\overline{\Gamma}_\eta(x)$. Finally $x\notin Y$ because $x\notin \text{Fix}(\Gamma,X)$. \\
	In case (b) we know that the group $\Delta_x$ is nilpotent of step $\leq S_0$ and it has index at most $I_0$ in $\overline{\Gamma}_\eta(x)$. We apply Lemma \ref{lemma-finite-index-generators} to find a generating set $\Sigma'$ of $\Delta_x$ such that each $\sigma' \in \Sigma'$ has length at most $(2I_0 - 1)$ in the alphabet $\lbrace \sigma_1,\ldots,\sigma_n\rbrace$. Since the lower central series of $\Delta_x$ has length $S\leq S_0$ we can find elements $\sigma_1',\ldots,\sigma_S' \in \Sigma'$ such that $[\sigma_1', [\sigma_2', \cdots[\sigma_{S-1}', \sigma_S']]]$ is not trivial and belongs to the center of $\Delta_x$. We call $g_0$ such an element and we notice that its length in the alphabet $\lbrace \sigma_1,\ldots,\sigma_n\rbrace$ is at most $(2I_0-1)(3\cdot 2^{S_0-1} - 2)$.\\
	We consider the finite set $\text{Conj}_{\overline{\Gamma}_\eta(x)}(g_0) = \lbrace hg_0h^{-1} \text{ s.t. } h\in \overline{\Gamma}_\eta(x)\rbrace$. Since  the center of $\Delta_x$ is characteristic in $\Delta_x$ and $\Delta_x$ is normal in $\overline{\Gamma}_\eta(x)$, then the center of $\Delta_x$ is normal in $\overline{\Gamma}_\eta(x)$. Hence all the elements of Conj$_{\overline{\Gamma}_\eta(x)}(g_0)$ are  in the center of $\Delta_x$ and commute.
	The isometries of Conj$_{\overline{\Gamma}_\eta(x)}(g_0)$ are of the same type of $g_0$. If $g_0$ is either hyperbolic or elliptic we take $Y=\bigcap_{h\in \overline{\Gamma}_\eta(x)}\text{Min}(hg_0h^{-1})$. It is non-empty by Proposition \ref{prop-commuting-isometries}, convex, closed and clearly $\overline{\Gamma}_\eta(x)$-invariant. Moreover by triangle inequality we have
	$$d(x,g_0x) \leq (2I_0-1)(3\cdot 2^{S_0-1} - 2) \cdot \eta = \lambda.$$
	This shows that $\ell(g_0) \leq \lambda$, so by definition $x\notin \text{Min}(g_0)$ and in particular $x\notin Y$. If $g_0$ is parabolic we take $\tau$ such that $\ell(g_0) < \tau < d(x,g_0x)$ and we consider the set $Y=\bigcap_{h\in \overline{\Gamma}_\eta(x)}M_\tau(hg_0h^{-1})$. It is again convex, closed, non-empty by Proposition \ref{prop-commuting-isometries} and $\overline{\Gamma}_\eta(x)$-invariant. Clearly $x\notin Y$ by definition of $\tau$.
	\vspace{3mm}
	
	\noindent \textbf{Step 5: the minumum $m$ of $\Psi$  is 0.}\\
	Let $Y$ be the subset of the previous step. We consider the projection $y$ of $x$ on $Y$ and we extend the geodesic segment $[y,x]$ to a geodesic ray $c$ beyond $x$. Since $Y$ is $\overline{\Gamma}_\eta(x)$-invariant then $d_{\sigma_i}$ is a non-decreasing function along $c$ for every $i=1,\ldots,n$. Let us suppose first they are all constant along $c$. This means that $\sigma_i c$ is parallel to $c$ for every $i=1,\ldots,n$. 	
	Therefore $g c$ is parallel to $c$ for all $g\in \overline{\Gamma}_\eta(x)$.\\
	In case (a), when $\Lambda_x$ is not trivial, we have the stronger conclusion $gc = c$ for all $g\in \Lambda_x$. In particular $x$ is fixed by all the elements of $\Lambda_x$, which is in contradiction to the fact that $x\notin \text{Min}_\lambda(\Gamma,X) \supseteq \text{Fix}(\Gamma,X)$. \\
	In case (b), when $g_0$ is parabolic this situation cannot happen. Indeed by our choice of $\tau$ in the definition of $Y$ we see that $g_0 c$ cannot be parallel to $c$. When $g_0$ is hyperbolic	we conclude, since $Y$ is contained in $\text{Min}(g_0)$, that also $x$ should belong to $\text{Min}(g_0)$ which is impossible because $\ell(g_0) \leq \lambda$ and $x\notin \text{Min}_\lambda(\Gamma,X)$.\\
	Therefore the quantity
	$$\bar{t} = \max\lbrace t\in [0,+\infty) \text{ s.t. } d_{\sigma_i}(c(t)) \leq d_{\sigma_i}(x) \,\,\forall i=1,\ldots,n \rbrace$$
	is finite. We denote by $\bar{x}$ the point $c(\bar{t})$. We claim that $\overline{\Sigma}_{\eta}(\bar{x}) = \overline{\Sigma}_{\eta}(x)$ and consequently $\Psi(\bar{x}) = \Psi(x) = m$. Let us define the set $A=\lbrace y \in [x,\bar{x}] \text{ s.t. } \overline{\Sigma}_{\eta}(y) = \overline{\Sigma}_{\eta}(x)\rbrace$. By Lemma \ref{lemma-basic-containment} and by definition of $\bar{t}$ we conclude that $A$ is open. Indeed for $y'$ in a open neighbourhood $U$ of $y$ we have $\overline{\Sigma}_\eta(y') \subseteq \overline{\Sigma}_\eta(y) = \overline{\Sigma}_\eta(x)$. By definition of $\bar{t}$ we know that in $c \cap U$, which is an open set of $[x,\bar{x}]$ the other inclusion holds.	
	Furthermore $A$ is closed: indeed suppose to have a sequence of points $y_n\in A$ converging to $y$. Therefore $\overline{\Sigma}_{\eta}(y) \supseteq \overline{\Sigma}_{\eta}(y_n) = \overline{\Sigma}_{\eta}(x)$ for $n$ big enough, again by Lemma \ref{lemma-basic-containment}. But this containement is actually an equality, by the maximality of $x$ with respect to $\preceq$ implying $y\in A$. Indeed observe that by definition $\Psi(y_n)=\Psi(x)$ for every $n$ and by continuity of $\Psi$ the same holds for $y$. Now we can use the maximality of $x$ with respect to $\preceq$ to conclude the equality.
	So $A = [x,\bar{x}]$, i.e. $\overline{\Sigma}_{\eta}(\bar{x}) = \lbrace \sigma_1,\ldots,\sigma_n\rbrace$.\\
	By Lemma \ref{lemma-basic-containment} it holds $\overline{\Sigma}_{\eta}(\bar{x}) \supseteq \overline{\Sigma}_{\eta}(c(\bar{t}+t))$ for every $t > 0$ small enough. By definition of $\bar{x}$ for such $t$'s there is at least one index $i\in \lbrace 1,\ldots,n\rbrace$ such that $d(\sigma_i c(\bar{t}+t), c(\bar{t}+t)) > d(\sigma_i \bar{x}, \bar{x})$. So
	$$\Psi(c(\bar{t} + t)) \leq \sum_{i=1}^n f(d(\sigma_i c(\bar{t}+t), c(\bar{t}+t))) < \sum_{i=1}^n f(d(\sigma_i \bar{x}, \bar{x})) = \Psi(\bar{x}) = m.$$
	Observe that the sum is made among the same set of isometries because of the inclusion $\overline{\Sigma}_{\eta}(\bar{x}) \supseteq \overline{\Sigma}_{\eta}(c(\bar{t}+t))$.
	This contradicts the fact that $m$ is the minimum, concluding the proof.
\end{proof}

\begin{obs} The proof of Theorem \ref{theo-alternative} in the almost-cocompact case. \\
 {\em
We gave the proof in the cocompact case   to avoid obscuring the main arguments with other technical details. However, the {\em same} proof works for almost cocompact actions (and in particular for finite volume ones):  the only steps where we used cocompactness are Step 2 and Step 3. Step 2 can be replaced in the almost-cocompact case by the following argument:\\
-- take again a sequence $x_n$ such that $\Psi(x_n) \to \inf \Psi$. The sequence $\Psi(x_n)$ is bounded, therefore it exists $0<\delta<\eta$ such that $d(x_n, gx_n)\geq \delta$ for every $g\in \Sigma_{\eta}(x_n)$ and for every $n$;  \\
-- this means that \textup{sys}$(\Gamma, x_n) \geq \delta$ for every $n$, so all the $x_n$'s are contained in the $\delta$-thick part \textup{Thick}$_{\delta}(\Gamma,X)$ of $X$, which is cocompact by almost cocompactness;\\
-- then, we can find a subsequence, called again $x_n$, that converges to $x_\infty$ and, since $\Psi$ is continuous, $\Psi(x_\infty) = \inf \Psi$.\\
Step 3 can be replaced by a similar argument: the map $x\mapsto \overline{\Sigma}_\eta(x)$ is still upper semi-continuous and $\Gamma$-invariant. The set $\Psi^{-1}(m)$ is contained in a thick part of $X$, as said above. By almost cocompactness there exists $L$ such that $\overline{\Sigma}_\eta(x) \leq L$ for all $x\in \Psi^{-1}(m)$. The rest of the proof do not change.
 }
\end{obs}

\begin{obs}
	\label{remark-union-interior} {\em
	An interesting, though simple, topological fact  is that when $X = \text{Min}_\lambda(\Gamma,X)$, since this family of closed subsets is locally finite then we can take the union only among those which have non-empty interior, as already observed in \cite{CCR01}, Proposition 0.2. 
	This is of particular interest for \emph{slim} actions of groups with torsion, for instance for actions on homology manifolds.}
\end{obs}

\section{Splitting and rigidity of CAT$(0)$ spaces under thin, slim actions}

Our aim is to find several rigidity results for spaces with small diastole. Let $I_0(P_0,r_0)$ be the constant given by Theorem \ref{theorem-Margulis-nilpotent} and let $J_0$ be the constant given by Bieberbach's Theorem in dimension 2: every discrete, cocompact group of isometries of $\mathbb{R}^2$ contains a lattice of index at most $J_0$.
We set 
\begin{equation}
	\label{eq-defin-lambda0}
	\lambda_0 = \min \left\lbrace\varepsilon_0, \frac{4r_0}{\max\lbrace I_0, J_0 \rbrace \cdot(P_0 + 1)}\right\rbrace
\end{equation}
and we call $\eta_0 = \eta_0(P_0,r_0,\lambda_0) =\eta_0(P_0,r_0) > 0$ the associated constant given by Theorem \ref{theo-alternative}.
Let $X$ be a complete, geodesically complete, CAT$(0)$-space which is $P_0$-packed at scale $r_0$ and let $\Gamma$ be a discrete, almost-cocompact group of isometries of $X$. We denote by $\Lambda$ the set of hyperbolic isometries $g$ of $\Gamma$ with $\ell(g)\leq \lambda_0$ and such that $\text{Min}(g)$ has non-empty interior. Clearly $h\Lambda h^{-1} = \Lambda$ for all $h\in \Gamma$. Theorem \ref{theo-alternative} and Remark \ref{remark-union-interior} implies that if the diastole is smaller than $\eta_0$ and the group is slim then $X=\bigcup_{g \in \Lambda}\text{Min}(g)$.

\vspace{1mm}

We start studying spaces whit points of dimension $1$:

\begin{prop}
	\label{prop-dimension-1}
	Let $X$ be a complete, geodesically complete, \textup{CAT}$(0)$-space which is $P_0$-packed at scale $r_0$ and  suppose that $X$ has a point of dimension $1$. 
	 Let $\Gamma$ be a discrete, slim, almost-cocompact  group of isometries of $X$.  If $\textup{dias}(\Gamma , X) < \eta_0$ then  $X$ is isometric to $\mathbb{R}$.

\end{prop}

\begin{proof}
	We have $X = \bigcup_{g\in \Lambda}\text{Min}(g).$
	Since $X$ has points of dimension $1$ then there is an open subset $U\subseteq X$ such that every point of $U$ has dimension $1$. In particular we can find a point $x\in X$ such that $\overline{B}(x,\varepsilon)$ is isometric to an interval for some $\varepsilon > 0$, by the discussion of Section \ref{subsec-cat(0)}. We take $g\in \Lambda$ such that $x\in \text{Min}(g)$. We can identify $x$ with a point $(y,0) \in D(g)\times \mathbb{R}$. Since $x$ is of dimension $1$ we conclude that $(y,0)$ is a point of dimension $1$, and therefore that $y$ is a point of dimension $0$: an isolated point. Since $D(g)$ is convex,  and in particular connected, we conclude that $D(g)$ is a point. Hence Min$(g)$ is isometric to $\mathbb{R}$.
	Suppose $X\neq \text{Min}(g)$. Then we can find a point $x'$ at distance $r_0$ from $\text{Min}(g)$. We consider the projection $y$ of $x'$ on $\text{Min}(g)$. We focus on the segment $[y, gy]$. Clearly some element in the orbit of $\langle g \rangle x$ must lie inside this segment. Therefore we can find a point, called again $x$, such that $x\in [y,gy]$ and $\overline{B}(x,\varepsilon)$ is isometric to an interval. We claim that the geodesic $[x',gx']$ has length at least $2r_0$. Let $c\colon [0,d(x',gx')] \to X$ be the geodesic $[x',gx']$. For every $t\in [0,d(x',gx')]$ we consider the geodesic $[c(t), x]$. There are only three cases: $[c(t)]_x = v_-$, $[c(t)]_x = v_+$ or $c(t) = x$, where $v_{\pm}$ are the only two elements of $\Sigma_x X$.
	It is easy to see that the sets $A_{\pm} = \lbrace t \in [0,d(x',gx')] \text{ s.t. } [c(t)]_x = v_{\pm}\rbrace$ are closed. If $c(t) \neq x$ for every $t\in [0,d(x',gx')]$ one concludes that $[0,d(x',gx')] = A_+$ or $[0,d(x',gx')] = A_-$. We claim that both these two cases are impossible since $[x']_x \neq [gx']_x$. Indeed we consider the convex map $d(x', \cdot)$ on the geodesic $\text{Min}(g)$ which has its minimum at $y$. This forces to have $[x']_x = [y]_x$. Indeed let $x_{\pm \varepsilon}$ be the points along Min$(g)$ at distance $\varepsilon$ from $x$ in direction $v_\pm$ and suppose $[y]_x = v_-$. Then $d(x_{-\varepsilon}, x') < d(x_{+\varepsilon}, x')$. The geodesic $[x',x]$ is either $[x',x_{-\varepsilon}] \cup [x_{-\varepsilon}, x]$ or $[x',x_{+\varepsilon}] \cup [x_{+\varepsilon}, x]$, so the inequality above says it is the first one. In other words $[x']_x = [y]_x$. In the same way $[gx']_x = [gy]_x$ and clearly $[y]_x \neq [gy]_x$. Therefore there must be some $t\in [0,d(x',gx')]$ such that $c(t)=x$.
	This forces the length of $c$ to be bigger than $2r_0$. Analogously we have that $d(g^nx', g^mx') > 2r_0$ for all distincts $n,m \in \mathbb{Z}$. Moreover $d(g^nx', y) \leq r_0 + \vert n \vert \ell(g) \leq r_0 + \vert n \vert \cdot \frac{4r_0}{P_0 + 1}$ for all $n\in \mathbb{Z}$. Therefore we can find $P_0 + 1$ distinct points of the form $g^n x'$ inside $\overline{B}(y,3r_0)$. They form a $2r_0$-separated subset of $\overline{B}(y,3r_0)$ and this is a contradiction to the packing assumption on $X$.
\end{proof}

Another application is a characterization of spaces  with small  diastole admitting a  strictly negatively curved open subset. It resembles \cite{CCR01}, Corollary 0.7.

\begin{cor}
	\label{corollary-CAT-epsilon}
	Let $X$ be a complete and geodesically complete  \textup{CAT}$(0)$-space which is $P_0$-packed at scale $r_0$, and suppose that  there exists some (arbitrarily small) open set $U \subset X$ which is \textup{CAT}$(-\varepsilon)$,  for some $\varepsilon > 0$. \linebreak  Let $\Gamma$ be a discrete, slim, almost-cocompact group of isometries of $X$:  if $\textup{dias}(\Gamma , X) < \eta_0$ then  $X$ is isometric to $\mathbb{R}$.
\end{cor}
\begin{proof}
	Again we know that $X=\bigcup_{g\in \Lambda}\text{Min}(g)$. If there exists a point of dimension $1$ then $X$ is isometric to $\mathbb{R}$ by Proposition \ref{prop-dimension-1}. Therefore we can suppose that every point of $X$ has dimension $\geq 2$. We then take a point $x\in U$ and an isometry $g\in \Lambda$, such that $x\in \text{Min}(g) = D(g)\times\mathbb{R}$. We identify $x$ with a point $(y,0)\in D(g)\times\mathbb{R}$. Since every point of $X$ has dimension $\geq 2$ then $D(g)$ is not a single point, i.e. there is $y\neq y' \in D(g)$. By the quadrangle flat theorem (\cite{BH09}, Theorem II.2.11) we have that the non-degenerate quadrangle $(y,0), g(y,0), g(y',0), (y',0)$ is isometric to a rectangle in $\mathbb{R}^2$. It is then possible to find inside $U$ a non-degenerate quadrangle isometric to a euclidean rectangle. This contradicts the assumption on $U$.
\end{proof}

	For the next applications we need the following splitting criterion.
	\begin{lemma}
		\label{lemma-splitting-cyclic}
		Let $X$ be a complete, geodesically complete, \textup{CAT}$(0)$-space which is $P_0$-packed at scale $r_0$. Let $\Gamma$ be a discrete, slim, almost-cocompact group of isometries of $X$. Suppose that $X=\bigcup_{g\in \Lambda} \textup{Min}(g)$ and that if $g,h \in \Lambda$ are such that $\textup{Min}(g) \cap \textup{Min}(h) \neq \emptyset$ then $\langle g,h \rangle$ is virtually cyclic.\\
		Then $X$ splits isometrically as $Y\times \mathbb{R}$. Moreover if $\Gamma$ is finitely generated then $\Gamma$ virtually splits as $\Gamma_Y \times \mathbb{Z}$, where $\Gamma_Y$ (resp. $\mathbb{Z}$) acts discretely on $Y$ (resp. $\mathbb{R}$).
	\end{lemma}
	The first part of the proof should be compared to \cite{BGS13}, Appendix 2, Lemma 1 and Lemma 2. 
	\begin{proof}
		A first observation is that if $K$ is a compact subset of $X$ then the set $\Lambda_K := \lbrace g\in \Lambda \text{ s.t. } \text{Min}(g)\cap K \neq \emptyset\rbrace$ is finite. Indeed if $x$ is a point of $K$ then $d(x,gx) \leq 2\text{Diam}(K) + \lambda_0$ for all $g \in \Lambda_K$ and the claim follows by discreteness of $\Gamma$.\\
		The second observation is that the isometries of $\Lambda$ have all parallel axes. Let $g,h \in \Lambda$ and take $x\in \text{Min}(g), y \in \text{Min}(h)$. Applying the first claim to $K=[x,y]$ and using the connectedness of $[x,y]$ we can find a finite sequence $g = f_0, f_1,\ldots, f_n = h$ of elements of $\Lambda$ such that $\text{Min}(f_i) \cap \text{Min}(f_{i+1}) \neq \emptyset$ for all $i = 0,\ldots, n-1$. We prove that the axes of $g$ and $h$ are parallel by induction on $n$. If $n=0$ it is trivial. If $n=1$, since $\langle g,h \rangle$ is virtually cyclic, there exist $p,q\in\mathbb{Z}^*$ such that $g^p = h^q$ and the claim follows. The induction step is similar: the axes of $g$ are parallel to the axes of $f_{n-1}$ and there exist $p,q\in\mathbb{Z}^*$ such that $f_{n-1}^p=h^q$.\\
		Let now $c$ be one axis of some $g\in \Lambda$ and consider the subset $P_c$ of $X$ of all geodesic lines parallel to $c$. It is closed, convex and $\Gamma$-invariant. Indeed if $\gamma \in \Gamma$ then $\gamma c$ is an axis of $\gamma g \gamma^{-1} \in \Lambda$, so $\gamma c$ is parallel to $c$. By minimality of the action, Proposition \ref{prop-minimal}, we conclude that $X=P_c$ and $P_c$ canonically splits isometrically as $Y \times \mathbb{R}$, cp. Theorem II.2.14 of \cite{BH09}.
		
		\noindent For the finitely generated case we need to show another claim: for every compact, convex $K \subseteq X$ there exists a hyperbolic isometry $\gamma \in \Gamma^*$ such that $K\subseteq \text{Min}(\gamma)$ and the axes of $\gamma$ are parallel to $c$. We know $\Lambda_K$ is finite, say $\Lambda_K = \lbrace g_1,\ldots,g_n\rbrace$, so $K\subseteq\bigcup_{i=1}^n\text{Min}(g_i)$. By connectedness of $K$ we know that for each $i$ there exists $j$ such that $\text{Min}(g_i) \cap \text{Min}(g_j) \neq \emptyset$. An easy computation shows that there are $p_i \in \mathbb{Z}^*$ such that $g_i^{p_i} = g_j^{p_j}$ for all $i,j \in \lbrace1,\ldots,n\rbrace$. If we call $\gamma$ this common power we have 
		$$K\subseteq\bigcup_{i=1}^n\text{Min}(g_i) \subseteq \bigcup_{i=1}^n\text{Min}(g_i^{p_i}) = \text{Min}(\gamma),$$
		and clearly the axes of $\gamma$ are parallel to $c$.\\
		Suppose $\Gamma$ is generated by $\lbrace h_1,\ldots,h_n\rbrace$ and set $R:= \max_{i=1,\ldots,n}d(x_0,h_ix_0)$ for some point $x_0 \in X$. We apply the claim above to $K = \overline{B}(x_0,2R)$, finding a hyperbolic isometry $\gamma \in \Gamma^*$ with axes parallel to $c$ such that $\overline{B}(x_0,2R) \subseteq \text{Min}(\gamma)$. 
		We recall that $\Gamma$ sends geodesic lines parallel to $c$ to geodesic lines parallel to $c$. In particular the axes of $h_i\gamma h_i^{-1}$ are parallel to the axes of $\gamma$. By definition of $R$ and $\gamma$ we observe that $B(h_ix_0,R) \subseteq \text{Min}(\gamma) \cap \text{Min}(h_i\gamma h_i^{-1})$. Therefore these two isometries have the same axis passing through every point of $B(h_ix_0,R)$ and they have the same translation length. This implies that $h_i\gamma h_i^{-1} = \gamma$ on $B(h_ix_0,R)$. Hence $h_i\gamma h_i^{-1} = \gamma$ because they coincide on an open subset. This shows that $\langle \gamma \rangle$ is normal, and even central, in $\Gamma$. So $\Gamma$ preserves $\text{Min}(\gamma) = X = Y \times \mathbb{R}$ because $\Gamma$ is minimal by Proposition \ref{prop-minimal} and $\Gamma$ virtually splits as $\Gamma_Y \times \mathbb{Z}$ by Theorem II.7.1 of \cite{BH09}, with each factor acting discretely on the corresponding space.
	\end{proof}

	\begin{proof}[Proof of Corollary \ref{cor-ZxZ}]  
		As usual we know that $X= \bigcup_{g \in \Lambda} \text{Min}(g)$. For every $x\in X$ let $\overline{\Gamma}_{\lambda_0}(x)$ be the non-trivial group generated by $\overline{\Sigma}_{\lambda_0}(x)$. By Proposition \ref{prop-structure-virt-nilpotent} we can find a convex subset $W_x \times \mathbb{R}^{k_x}$ and a finite index subgroup $\Gamma_{W_x} \times \mathbb{Z}^{k_x}$ of $\overline{\Gamma}_{\lambda_0}(x)$, where $\Gamma_{W_x}$ acts on $W_x$ without hyperbolic elements and $\mathbb{Z}^{k_x}$ acts on $\mathbb{R}^{k_x}$ as a lattice. The existence of a hyperbolic isometry displacing $x$ by at most $\lambda_0$ forces to have $k_x\geq 1$.
		If $k_x\geq 2$ we have the desired $\mathbb{Z}\times \mathbb{Z}$, so we can suppose $k_x=1$. If $\Gamma_{W_x}$ contains some parabolic element we have again $\mathbb{Z} \times \mathbb{Z}$ inside $\Gamma$. In the remaining case we have $k_x=1$ and $\Gamma_{W_x}$ finite, so $\overline{\Gamma}_{\lambda_0}(x)$ is virtually cyclic for every $x\in X$.
		The assumptions of Lemma \ref{lemma-splitting-cyclic} are clearly satisfied, so $X$ splits as $Y\times \mathbb{R}$. If moreover $\Gamma$ is cocompact (hence finitely generated) then it has a finite index subgroup which splits as $\Gamma_Y \times \mathbb{Z}$, and clearly $\Gamma_Y$ acts cocompactly on $Y$. By \cite{Swe99}, Theorem 11 either $Y$ is a point or $\Gamma_Y$ contains an infinite order element, because $Y$ is also geodesically complete. In the second case we have again $\mathbb{Z} \times \mathbb{Z}$ inside $\Gamma$, in the first one $X=\mathbb{R}$.

	\end{proof}
	
	\begin{obs}
		\em{If in Corollary \ref{cor-ZxZ} we consider finitely generated almost-cocompact groups in place of cocompact ones we can still apply Lemma \ref{lemma-splitting-cyclic} to get a finite index subgroup of the form $\Gamma_Y \times \mathbb{Z}$. Moreover we could get the same conclusion of Corollary \ref{cor-ZxZ} under these weaker assumptions if we knew that $\Gamma_Y$ contains an isometry of infinite order. This last statement is the content of the following conjecture.}
	\end{obs}
	\begin{conj}
		Let $X$ be a proper, \textup{CAT}$(0)$-space and let $\Gamma$ be a discrete group of isometries of $X$, not necessarely cocompact. If $\Gamma$ is torsion, i.e. it contains only finite order elements, then $\Gamma$ is finite.\\
		An easier version for us would be: if $\Gamma$ is almost-cocompact then it contains an infinite order element.
	\end{conj}

By the flat torus theorem  (\cite{BH09}, Theorem I.7.1), if $\Gamma$ contains $\mathbb{Z}\times\mathbb{Z}$ then $X$ contains an isometrically embedded flat $\mathbb{R}^2$. The same holds if $X$ splits as $Y\times \mathbb{R}$, with $Y$ different from a point. This is impossible for instance if $X$ is a visibility space (see Definition I.9.28 of \cite{BH09}), therefore we obtain:

\begin{cor}
	\label{cor-visibility}
	Let $X$ be a complete, geodesically complete, \textup{CAT}$(0)$-space which is $P_0$-packed at scale $r_0$, and assume that  $X$ is a visibility space. Let $\Gamma$ be a discrete, slim, almost-cocompact group of isometries of $X$:
	if $\textup{dias}(\Gamma , X) <\eta_0$ then $X$ is isometric to $ \mathbb{R}$.
\end{cor}

The last example of characterization  of spaces  with small  diastole   is the following, holding for spaces of dimension at most $2$.
\begin{prop}
	\label{prop-dimension-2}
	Let $X$ be a complete, geodesically complete, \textup{CAT}$(0)$-space which is $P_0$-packed at scale $r_0$, and assume that the dimension of $X$ is at most  $2$. Let $\Gamma$ be a discrete, slim, almost-cocompact group of isometries of $X$: if  $\textup{dias}(\Gamma , X) < \eta_0$ then $X$ is isometric to a product $T \times \mathbb{R}$, where $T$ is a geodesically complete, simplicial tree (possibily reduced to a point). Moreover if $\Gamma$ is cocompact then it has a finite index subgroup which splits as $\Gamma_T \times \mathbb{Z}$, where $\Gamma_T$ (resp. $\mathbb{Z}$) acts discretely and cocompactly on $T$ (resp. $\mathbb{R}$).
 \end{prop}

First of all we need to characterize trees as those CAT$(0)$-spaces with geometric dimension $1$.

\begin{lemma}[See also the proof of Lemma 6.3 of \cite{OP21}]
	\label{lemma-tree-recognition}
	Let $Y$ be a \textup{CAT}$(0)$-space such that $\Sigma_y Y$ is discrete and not empty for every $y \in Y$, in other words $\textup{GD}(Y) = 1$. Then $Y$ is a real tree, i.e. for every two points $x,z \in Y$ the image of any continuous injective map $\alpha \colon [a,b] \to X$ such that $\alpha(a)=x$ and $\alpha(b)=z$ coincides with the image of $[x,z]$. 
	Moreover, if $Y$ is geodesically complete and proper then it is a simplicial tree.
\end{lemma}
\begin{proof}
	Let $x,z \in Y$, $c = [x,z]$ and $\alpha$ be a continuous injective map $\alpha \colon [a,b] \to Y$ such that $\alpha(a)=x$ and $\alpha(b)=z$. We want to show that $\text{Im}(\alpha) = \text{Im}(c)$, where $c\colon [0,d(x,z)] \to Y$ is the geodesic $[x,z]$. We first prove that $\text{Im}(c) \subseteq \text{Im}(\alpha)$. Let $s \in (0,d(x,z))$. The set $\Sigma_{c(s)}Y$ is discrete, say equal to $\lbrace v_i\rbrace_{i\in I}$. We consider the sets 
	$A_i = \lbrace t\in [a,b] \text{ s.t. } [\alpha(t)]_{c(s)} = v_i \rbrace$, for $i \in I$. Observe that the quantity $[\alpha(t)]_{c(s)}$ is defined as soon as $\alpha(t) \neq c(s)$. By discreteness of $\Sigma_{c(s)}Y$ it is easy to check that each $A_i$ is open and closed. By connectedness of $[a,b]$ there are two possibilities: either there exists $i\in I$ such that $A_i = [a,b]$ or there is some $t \in (a,b)$ such that $\alpha(t) = c(s)$. The first possibility has to be excluded since clearly $[\alpha(a)]_{c(s)} \neq [\alpha(b)]_{c(s)}$, so $c(s) \in \text{Im}(\alpha)$.\\
	It remains to prove the other inclusion. 
	We set $A = \lbrace t\in [a,b] \text{ s.t. } \alpha(t) \in \text{Im}(c)\rbrace$. It is a closed subset of $[a,b]$, so the complementary is the disjoint union of countably many intervals, $A^c = \bigcup_{i=1}^N (a_i,b_i)$, $N\in \mathbb{N}\cup\lbrace \infty \rbrace$. In the interval $(a_i,b_i)$ we have $\text{Im}(\alpha\vert_{[a_i,b_i]})\cap c = \lbrace \alpha(a_i), \alpha(b_i)\rbrace$. By injectivity of $\alpha$ we know that $\alpha(a_i) \neq \alpha(b_i)$, so the subsegment of $c$ joining these two points is not trivial. However applying the first part of the proof to the continuous map $\alpha\vert_{[a_i,b_i]}$ we obtain $\text{Im}(\alpha\vert_{[a_i,b_i]})\cap c \supseteq [\alpha(a_i), \alpha(b_i)]$, a contradiction. This shows that $A=[a,b]$. \\
	If a real tree $Y$ is locally compact and geodesically complete then it is simplicial by the description of geodesically complete, proper, CAT$(0)$-spaces of \cite{LN19} recalled in Section \ref{subsec-cat(0)}.
\end{proof}
The second tool we need is a better description of the almost stabilizers in dimension at most $2$.
	\begin{lemma}
		\label{lemma-bieberbach-dim-2}
		Let $X$ be a complete, geodesically complete, \textup{CAT}$(0)$-space which is $P_0$-packed at scale $r_0$, and assume that the dimension of $X$ is at most $2$. Let $\Gamma$ be a discrete group of isometries of $X$ and let $\varepsilon_0$ be the Margulis constant given by Theorem \ref{theorem-Margulis-nilpotent}. Then for every $x\in X$ one of the following three mutually exclusive possibilities occur:
		\begin{itemize}
			\item[(i)] $\overline{\Gamma}_{\varepsilon_0}(x)$ has no hyperbolic isometries;
			\item[(ii)] $\overline{\Gamma}_{\varepsilon_0}(x)$ has no parabolic isometries and it is virtually $\mathbb{Z}$;
			\item[(iii)] $\overline{\Gamma}_{\varepsilon_0}(x)$ has no parabolic isometries and it has a subgroup of index at most $J_0$ which is isomorphic to $\mathbb{Z}^2$.
		\end{itemize}
	\end{lemma}
	\begin{proof}
		As $\overline{\Gamma}_{\varepsilon_0}(x)$ is virtually nilpotent, by Proposition \ref{prop-structure-virt-nilpotent} we can find a $\overline{\Gamma}_{\varepsilon_0}(x)$-invariant convex subset $W\times \mathbb{R}^k$ of $X$ and a finite index subgroup of $\overline{\Gamma}_{\varepsilon_0}(x)$ of the form $\Gamma_W \times \mathbb{Z}^k$, where $\Gamma_W$ acts on $W$ by parabolic or elliptic isometries. If there is at least one hyperbolic isometry in $\overline{\Gamma}_{\varepsilon_0}(x)$ then $k\geq 1$ and of course $k\leq 2$. We start with the case $k=2$. By dimensional reasons $W$ must have dimension $0$, so it must be a point. This means that $\overline{\Gamma}_{\varepsilon_0}(x)$ acts on $\mathbb{R}^2$ cocompactly and discretely because it contains a lattice of $\mathbb{R}^2$ as finite index subgroup. By Bieberbach's Theorem we conclude that $\overline{\Gamma}_{\varepsilon_0}(x)$ has a subgroup of index at most $J_0$ which is isomorphic to $\mathbb{Z}^2$.\\
		Now suppose $k=1$. Then $W$ is a CAT$(0)$-space of geometric dimension at most $1$, so it is a real tree by Lemma \ref{lemma-tree-recognition} or a point. In any case $W$ has no parabolic isometries, so $\Gamma_W$ contains only elliptic elements. As a consequence $\Gamma_W$ is finite, so $\overline{\Gamma}_{\varepsilon_0}(x)$ has no parabolic isometries and it is virtually cyclic. 
	\end{proof}

\begin{proof}[Proof of Proposition \ref{prop-dimension-2}]
	We know that $X=\bigcup_{g\in \Lambda}\text{Min}(g)$. If $X$ has a point of dimension $1$ then  $X=\mathbb{R}$ by Proposition \ref{prop-dimension-1}. We can therefore suppose that every point of $X$ is of dimension $2$. 
	The minimal set of $g\in \Lambda$ is of the form $\text{Min}(g) = D(g) \times \mathbb{R}$. Since every point of $X$ has dimension $2$ and $\text{Min}(g)$ has non-empty interior we deduce that $D(g)$ is not a point, and by Lemma \ref{lemma-tree-recognition} it is a (maybe compact) real tree. If $g,h \in \Lambda$ are such that $\text{Min}(g) \cap \text{Min}(h) \neq \emptyset$ then either $\langle g,h \rangle$ is virtually cyclic or $\langle g^{J_0},h^{J_0} \rangle \cong \mathbb{Z}^2$ by Lemma \ref{lemma-bieberbach-dim-2}. Let us study the second case. Since $g^{J_0}$ and $h^{J_0}$ commute, the isometry $h^{J_0}$ acts on $\text{Min}(g^{J_0}) = D(g^{J_0}) \times \mathbb{R}$ preserving the product decomposition. Let $h'$ be the isometry of   $D(g^{J_0})$   induced by  $h^{J_0}$. The isometry $h'$ cannot have fixed points otherwise $h^{J_0}$ and $g^{J_0}$ have a common axis, and they would generate a virtually cyclic group. Therefore $h'$ is a hyperbolic isometry of $D(g^{J_0})$ and we call $c$ its axis. We want to show that $D(g^{J_0}) = c$. Suppose there exists a point $x\in D(g^{J_0})$ which is not in $c$. Let $y$ be its projection on $c$. As $D(g^{J_0})$ is a tree, the geodesic in $D(g^{J_0})$, and so in $X$, between $x$ and $h'x$ is the concatenation $[x,y] \cup [y,h'y] \cup [h'y,h'x]$. We extend the geodesic segment $[y,x]$ in $X$ beyond $x$, defining a geodesic ray $c_y$. We do the same for the geodesic segment $[h'y,h'x]$ defining a geodesic ray $c_{h'y}$. We observe that the concatenation $c_y \cup [y,h'y] \cup c_{h'y}$ is a geodesic line in $X$. In particular the distance between the two points along $c_y$ and $c_{h'y}$ at distance $r_0$ from $y$ and $h'y$ is at least $2r_0$. We can repeat the construction for all points $(h')^n x$, extending the geodesic segments $[(h')^ny, (h')^n x]$ beyond $(h')^nx$ and finding points that are at distance at least $2r_0$ one from the other. Clearly $\ell(h') \leq \ell(h^{J_0}) \leq J_0\cdot\lambda_0$, so arguing as in the proof of Proposition \ref{prop-dimension-1} we conclude we can find more than $P_0$ points inside $\overline{B}(y,3r_0)$ that are $r_0$-separated as soon as $\lambda_0 \leq \frac{4r_0}{J_0(P_0+1)}$ which is the case by \eqref{eq-defin-lambda0}. This violates the packing assumption on $X$, so $D(g^{J_0}) = \mathbb{R}$ and Min$(g^{J_0})$ must be isometric to $\mathbb{R}^2$. The same conclusion holds for Min$(h^{J_0})$. By the flat torus theorem we know that 
	$$Y \times \mathbb{R}^2 = \text{Min}(g^{J_0}) \cap \text{Min}(h^{J_0}) \subseteq \text{Min}(g^{J_0}) = \mathbb{R}^2$$
	for some convex subset $Y$ of $X$. Since the dimension of $X$ is at most $2$ the set $Y$ must be a point and the inclusion above is an equality. This shows that $\text{Min}(g^{J_0}) = \text{Min}(h^{J_0}) = \mathbb{R}^2$. The same argument shows that also $\text{Min}(g^{pJ_0}) = \text{Min}(g^{J_0}) = \mathbb{R}^2$ for all $p\neq 0$, because $h^{J_0}$ acts on $\text{Min}(g^{pJ_0})$. In the same way we have $\text{Min}(h^{pJ_0}) = \text{Min}(h^{J_0}) = \mathbb{R}^2$ for all $p\neq 0$.\\
	We claim that $\text{Min}(f) \subseteq \mathbb{R}^2$ for all $f\in \Lambda$. It is not difficult to see it by taking a path $g = f_0, f_1,\ldots, f_n = f$ such that $f_i \in \Lambda$ and $\text{Min}(f_i) \cap \text{Min}(f_{i+1}) \neq \emptyset$ for all $i=0,\ldots,n-1$. Therefore $X = \bigcup_{g \in \Lambda}\text{Min}(g) = \mathbb{R}^2$ which is $\mathbb{R} \times \mathbb{R}$. If moreover $\Gamma$ is cocompact then the thesis is implied by Bieberbach's Theorem in dimension $2$.\\
	It remains the case where $\langle g,h \rangle$ is virtually cyclic for all $g,h \in \Lambda$ such that $\text{Min}(g)\cap \text{Min}(h) \neq \emptyset$. An application of Lemma \ref{lemma-splitting-cyclic} gives the splitting of $X$ as $T\times \mathbb{R}$. Here $T$ is necessarely a simplicial tree because of Lemma \ref{lemma-tree-recognition}. If $\Gamma$ is cocompact then the thesis follows from the second part of Lemma \ref{lemma-splitting-cyclic}.

\end{proof}

\bibliographystyle{alpha}
\bibliography{Thin_actions_on_CAT(0)-spaces}

 \end{document}